  \documentclass[letterpaper, oneside]{article}

  \usepackage[T1]{fontenc}
  \usepackage{fourier}

  \usepackage{geometry}
  \usepackage{amsmath,amsthm,amssymb}
  \usepackage{rotating}
  \usepackage[all]{xy}
  \usepackage{tikz}
  \usepackage{xspace}

  \newcommand\address[1]{\newcommand\addressaux{#1}}
  \newcommand\email[1]{\newcommand\emailaux{#1}}
  \newcommand\subjclass[1]{\newcommand\subjclassaux{#1}}
  \newcommand\keywords[1]{\newcommand\keywordsaux{#1}}
  \makeatletter
  \renewcommand*{\author}[1]{\gdef\@author{#1$^\dagger$}}
  \makeatother
  \newcommand\mydate{\date{\small\itshape\today
    \footnotetext{2010 {\itshape Mathematics Subject Classification}. Primary \subjclassaux.}
    \footnotetext{{\itshape Key words and phrases}. \keywordsaux.}
    \footnotetext{$^\dagger${\itshape Current address}: \addressaux.}
    \footnotetext{$^\dagger${\itshape E-mail address}: \tt\emailaux.}
    }}

    \def\newmcommand#1[#2]#3{\newcommand{#1}[#2]{\ensuremath{#3}\xspace}}
    \def\renewmcommand#1[#2]#3{\renewcommand{#1}[#2]{\ensuremath{#3}\xspace}}

    \newmcommand   \M         [0] { M                                  }
    \newmcommand   \G         [0] { G                                  }
    \newmcommand   \C         [0] { \mathcal{C}                        }

    \newmcommand   \detvar    [1] { D^{#1}                             }
    \newmcommand   \detvark   [0] { D^{k}                              }
    \newmcommand   \detvari   [0] { D^{i}                              }
    \newmcommand   \detvarj   [0] { D^{j}                              }

    \newmcommand   \jet       [2] { #2_{#1}                            }
    \newmcommand   \arc       [1] { \jet{\infty}{#1}                   }
    \newmcommand   \jetn      [1] { \jet{n}{#1}                        }
    \newmcommand   \jetm      [1] { \jet{m}{#1}                        }
    \newmcommand   \jetp      [1] { \jet{p}{#1}                        }
    \newmcommand   \jetq      [1] { \jet{q}{#1}                        }
    \newmcommand   \jeti      [1] { \jet{i}{#1}                        }
    \newmcommand   \jetj      [1] { \jet{j}{#1}                        }

    \newmcommand   \pjet      [2] { \left({#2}\right)_{#1}             }
    \newmcommand   \parc      [1] { \pjet{\infty}{#1}                  }
    \newmcommand   \pjetn     [1] { \pjet{n}{#1}                       }
    \newmcommand   \pjetp     [1] { \pjet{p}{#1}                       }
    \newmcommand   \pjetq     [1] { \pjet{q}{#1}                       }
    \newmcommand   \pjeti     [1] { \pjet{i}{#1}                       }
    \newmcommand   \pjetj     [1] { \pjet{j}{#1}                       }

    \newmcommand   \pt        [1] { \Lambda_{#1}                       }
    \newmcommand   \ept       [1] { \overline{\Lambda}_{#1}            }
    \newmcommand   \bpt       [2] { \Lambda_{#1,#2}                    }
    \newmcommand   \ptr       [0] { \pt{r}                             }
    \newmcommand   \eptr      [0] { \ept{r}                            }
    \newmcommand   \bptrn     [0] { \bpt{r}{n}                         }

    \newmcommand   \orb       [1] { \mathcal{C}_{#1}                   }
    \newmcommand   \torb      [2] { \mathcal{C}_{#1,#2}                }

    \newmcommand   \cont      [2] { \operatorname{Cont}^{#1}( #2 )     }
    \newmcommand   \contp     [1] { \cont{p}{#1}                       }

    \newmcommand   \dis       [2] { k_{#1}(#2)                         }

    \renewmcommand \AA        [0] { \mathbf{A}                         }
    \newmcommand   \PP        [0] { \mathbf{P}                         }
    \newmcommand   \RR        [0] { \mathbf{R}                         }
    \newmcommand   \QQ        [0] { \mathbf{Q}                         }
    \newmcommand   \CC        [0] { \mathbf{C}                         }
    \newmcommand   \ZZ        [0] { \mathbf{Z}                         }
    \newmcommand   \NN        [0] { \mathbf{N}                         }
    \newmcommand   \NNb       [0] { \overline{\mathbf{N}}              }
    \newmcommand   \LL        [0] { \mathbf{L}                         }
    \newmcommand   \II        [0] { \mathcal{I}                        }
    \newmcommand   \OO        [0] { \mathcal{O}                        }

    \newmcommand   \GL        [1] { \operatorname{GL}_{#1}             }
    \newmcommand   \GLr       [0] { \GL{r}                             }
    \newmcommand   \GLs       [0] { \GL{s}                             }

    \newmcommand   \psr       [1] { {#1}[\![t]\!]                      }
    \newmcommand   \tpsr      [2] { {#2}[t]/(t^{#1+1})                 }

    \newcommand    \Mustata   [0] { Musta\c{t}\u{a}\xspace             }

    \DeclareMathOperator \Hom  {Hom}
    \DeclareMathOperator \Spec {Spec}
    \DeclareMathOperator \ord  {ord}
    \DeclareMathOperator \Jac  {Jac}
    \DeclareMathOperator \lct  {lct}

    \theoremstyle{plain}

      \newtheorem{introtheorem}{Theorem}

      \newtheorem{theorem}{Theorem}[section]
      \newtheorem{proposition}[theorem]{Proposition}
      \newtheorem{lemma}[theorem]{Lemma}
      \newtheorem{corollary}[theorem]{Corollary}

    \theoremstyle{definition}

      \newtheorem{definition}[theorem]{Definition}
      \newtheorem{remark}[theorem]{Remark}
      \newtheorem{notation}[theorem]{Notation}
      \newtheorem{example}[theorem]{Example}

    \newcommand\diagramexample{
      \[
         \def\ydinit{\path (0,0) coordinate (P);}
         \def\ydr##1{%
           \path (P) coordinate (Q);
           \foreach \i in {1,...,##1}%
           {%
             \draw (Q) rectangle +(1,-1);%
             \path (Q) +(1,0) coordinate (Q);%
           }%
           \path (P) +(0,-1) coordinate (P);%
         }
         \def\ydi{\ydr 6 
           \draw (Q) -- +(1,0); 
           \draw (Q) +(0,-1) -- +(1,-1);
           \draw (Q) +(1.5,-.45) node {\dots};
         }
         \def\yd##1{\raisebox{-3ex}{\begin{tikzpicture}[scale=.3]
           \ydinit ##1\end{tikzpicture}}}
         (5,4,3,3,2) \quad = \quad 
         \yd{\ydr 2 \ydr 3 \ydr 3 \ydr 4 \ydr 5}
         \qquad\qquad\qquad
         (\infty,\infty,4,2,1) \quad = \quad 
         \yd{\ydr 1 \ydr 2 \ydr 4 \ydi \ydi}
      \]
    }

    \newcommand\partitionsfigure{
      \begin{figure}[ht]
         \def\ydinit{\path (0,0) coordinate (P);}
         \def\ydr##1{%
           \path (P) coordinate (Q);
           \foreach \i in {1,...,##1}%
           {%
             \draw (Q) rectangle +(1,-1);%
             \path (Q) +(1,0) coordinate (Q);%
           }%
           \path (P) +(0,-1) coordinate (P);%
         }
         \def\ydi{\ydr 6 
           \draw (Q) -- +(1,0); 
           \draw (Q) +(0,-1) -- +(1,-1);
           \draw (Q) +(1.5,-.45) node {\dots};
         }
         \def\yd##1{\raisebox{-3ex}{\begin{tikzpicture}[scale=.3]
           \ydinit ##1\end{tikzpicture}}}
        \[
          \begin{array}{lcl}
            \\
            \yd{
              \ydr 2 \ydr 4 \ydr 4 \ydi
              \draw[fill=black!20] (4,-2) rectangle +(1,-1);
              \draw (7,-0) rectangle +(1,-1);
              \draw [<-] (7.5,-0.5) to [out=-90,in=0] (4.5,-2.5);
            } &
            \qquad \qquad \qquad \qquad &
            \yd{
              \fill[fill=black!20] (4,-2) rectangle +(3,-1);
              \ydr 2 \ydr 4
              \ydi \ydi 
              \draw (7,-0) rectangle ++(1,-1) rectangle ++(1,1);
              \draw (9,-0) -- (10,-0); \draw (9,-1) -- (10,-1);
              \draw (10.5,-0.55) node {\dots};
              \draw [<-] (7.5,-0.5) to [out=180,in=90] (4.5,-2.5);
            } \\\\
            \text{Single removal}&&
            \text{Infinite removal}
            \\\\
            \yd{
              \fill[fill=black!20] (1,-2) rectangle +(1,-1);
              \ydr 1 \ydr 1 \ydr 2 \ydr 5 \ydr 6
              \draw [<-,rounded corners=2pt] 
                (4.5,-3.5) -- (4.5, -2.5) -- (1.5,-2.5);
            }&
            \qquad \qquad \qquad \qquad &
            \yd{
              \fill[fill=black!20] (2,-1) rectangle +(1,-1);
              \ydr 2 \ydr 3 \ydr 3 \ydr 3 \ydr 4 \ydr 5
              \draw [<-,rounded corners=2pt] 
                (3.5,-4.5) -- (3.5, -1.5) -- (2.5,-1.5);
            }\\\\
            \text{Slip}&&
            \text{Fall}
          \end{array}
        \]
      \end{figure}
    }

    \def\RCS$#1: #2 ${\expandafter\def\csname RCS#1\endcsname{#2}}
    \def\RCSdd|#1#2#3#4#5#6#7#8#9|{%
      \begingroup %
      \day #9 %
      \month #6#7 %
      \year #1#2#3#4  %
      \xdef\RCSDate{\today} %
      \endgroup %
      }
    \def\RCSd$#1: #2 #3:#4:#5 ${%
      \RCSdd|#2| %
      \def\RCSTime{#3:#4} %
      }

    \RCS$Revision: 3.14 $
    \RCSd$Date: 2011-02-06 23:23:55-07 $

\begin{document}

  %\renewcommand{\month}{12} 
  %\renewcommand{\day}{5} 
  %\renewcommand{\year}{2009}

%section{Title and Abstract}

  \title     {Arcs on Determinantal Varieties}
  \author    {Roi Docampo}
  \address   {Department of Mathematics,
              University of Utah,
              155 S 1400 E RM 233,
              Salt Lake City, UT 84112, USA} 
  \email     {docampo@math.utah.edu}

  \mydate

  \maketitle

  \begin{abstract} 
  We study arc spaces and jet schemes of generic determinantal varieties.
  Using the natural group action, we decompose the arc spaces into orbits, and
  analyze their structure. This allows us to compute the number of irreducible
  components of jet schemes, log canonical thresholds, and topological zeta
  functions. 
  \end{abstract}

\section*{Introduction}

  Let $\M=\AA^{r s}$ denote the space of $r\times s$ matrices, and assume that
  $r \le s$. Let $\detvark \subset \M$ be the \emph{generic determinantal
  variety} of rank $k$, that is, the subvariety of \M whose points correspond
  to matrices of rank at most $k$. The purpose of this paper is to analyze the
  structure of arc spaces and jet schemes of generic determinantal varieties.

  Arc spaces and jet schemes have attracted considerable attention in recent
  years.  They were introduced to the field by J.\ F.\ Nash \cite{Nas95}, who
  noticed for the fist time their connection with resolution of singularities.
  A few years later, M.\ Kontsevich introduced motivic integration
  \cite{Kon95, DL99}, popularizing the use of the arc space.  And starting
  with the work of M.$\,\!$\Mustata, arc and jets have become a standard tool in
  birational geometry, mainly because of their role in formulas for
  controlling discrepancies \cite{Mus01, Mus02, EMY03, ELM04, EM06, dFEI08}.
  
  But despite their significance from a theoretical point of view, arc spaces
  are often hard to compute in concrete examples. The interest in Nash's
  conjecture led to the study of arcs in isolated surface singularities
  \cite{LJ90, Nob91, LJR98, Ple05, PPP06, LJR08}. Quotient singularities are
  analyzed from the point of view of motivic integration in \cite{DL02}. We
  also understand the situation for monomial ideals \cite{GS06, Yue07b} and
  for toric varieties \cite{Ish04}. But beyond these cases very little is
  known about the geometry of the arc space of a singular algebraic variety.
  The purpose of this article is to analyze in detail the geometric structure
  of arc spaces and jet schemes of generic determinantal varieties, giving a
  new family of examples for which the arc space is well understood.

  Recall that arcs and jets are higher order analogues of tangent vectors.
  Given a variety $X$ defined over \CC, an arc of $X$ is a \psr\CC-valued
  point of $X$, and an $n$-jet is a $\tpsr n\CC$ -valued point. A $1$-jet is
  the same as a tangent vector.  Just as in the case of the tangent space,
  arcs on $X$ can be identified with the closed points of a scheme \arc X,
  which we call the \emph{arc space} of $X$ \cite{Nas95,Voj07}.  Similarly,
  $n$-jets give rise to the \emph{$n$-th jet scheme} of $X$, which we denote
  by \jetn X (see Section \ref{sec:introarcs} for more details).

  For the space of matrices $\M$, the arc space \arc\M and the jet scheme
  \jetn\M can be understood set-theoretically as the spaces of matrices with
  entries in the rings \psr\CC and $\tpsr n\CC$, respectively. \arc\detvark
  and \jetn\detvark are contained in \arc\M and \jetn\M, and their equations
  are obtained by ``differentiating'' the $k \times k$ minors of a matrix of
  independent variables. We approach the study of \arc\detvark and
  \jetn\detvark with three goals in mind: understand the \emph{topology} of
  \jetn\detvark, compute \emph{log canonical thresholds} for the pairs $(\M,
  \detvark)$, and compute \emph{topological zeta functions} for $(\M,
  \detvark)$.

  \subsection{Irreducible components of jet schemes}

  The topology of the jet schemes \jetn\detvark is intimately related to the
  \emph{generalized Nash problem}.

  Given an irreducible family of arcs $\C \subset \arc\M$, we can consider
  $\nu_\C$, the order of vanishing along a general element of $\C$. The
  function $\nu_\C$ is almost a discrete valuation, the only problem being
  that it takes infinite value on those functions vanishing along all the arcs
  in $\C$. If there are no such functions, we call the family \emph{fat}
  \cite{Ish06}, and we see that irreducible fat families of arcs give rise to
  discrete valuations.

  Conversely, given a divisorial valuation $\nu$ over $\M$, any isomorphism
  from $\psr\CC$ to the completion of the valuation ring produces a non-closed
  point of $\arc \M$.  The closure of any these points can be easily seen to
  give an irreducible fat family of arcs inducing $\nu$.

  Among all closed irreducible fat families of arcs inducing a given
  divisorial valuation, there exists a maximal one with respect to the
  order of containment, known as the \emph{maximal divisorial set} (see
  Section \ref{sec:introarcs} for details). In this way we get a bijection
  between divisorial valuations and maximal divisorial sets in the arc
  space, and we can use the topology in the arc space to give structure to
  the set of divisorial valuations. More concretely, the containment of
  maximal divisorial sets induces a partial order on valuations. The
  understanding of this order is known as the \emph{generalized Nash
  problem} \cite{Ish06}.

  There are other ways to define orders in the set of divisorial valuations.
  For example, thinking of valuations as functions on $\OO_\M$, we can
  partially order them by comparing their values. In dimension two, the
  resolution process also gives an order. It can be shown that the order
  induced by the arc space is different from any previously known order
  \cite{Ish06}, but beyond that, not much is known about the generalized Nash
  problem. A notable exception is the case of of toric valuations on toric
  varieties, which was studied in detail in \cite{Ish04}.

  %The topology of the jet schemes \jetn\detvark is intimately related to the
  %\emph{generalized Nash problem}. Given a divisorial valuation $\nu$ over \M,
  %there is natural way to define a family of arcs in \arc\M, known as the
  %\emph{maximal divisorial set} associated to $\nu$ (see Section
  %\ref{sec:introarcs} for details). Containment of sets in \arc\M induces an
  %order on the set of divisorial valuations, the understanding of which is
  %known as the \emph{generalized Nash problem} \cite{Ish06}. At the moment,
  %very little is known about this problem for arbitrary varieties.

  Determining the irreducible components for \jetn\detvark is essentially
  equivalent to computing minimal elements among those valuations over \M that
  satisfy certain contact conditions with respect to \detvark. In
  Section~\ref{orbposet} we solve the generalized Nash problem for
  \emph{invariant} divisorial valuations, and use this to prove the following
  theorem.

  \begin{introtheorem}\label{introthm1}
    Let \detvark be the determinantal variety of matrices of size $r\times s$
    and rank at most $k$, where $k < r\le s$.  Let $\jetn\detvark$ be the $n$-th
    jet scheme of \detvark.  If $k=0$ or $k=r-1$, the jet scheme \jetn\detvark
    is irreducible. Otherwise the number of irreducible components of
    $\jetn\detvark$ is
    \[
      n+2 - \left\lceil \frac{n+1}{k+1} \right\rceil.
    \]
  \end{introtheorem}

  Jet schemes for determinantal varieties were previously studied in
  \cite{KS05a, KS05b, Yue07a}. Up to now, the approach has always been to use
  techniques from commutative algebra, performing a careful study of the
  defining equations. This has been quite successful for ranks $1$ and $r-1$,
  especially for square matrices, but the general case seems too complex for
  these methods.

  Our approach is quite different in nature: we focus on the natural group
  action. This is a technique that already plays a central role in Ishii's
  study of the arc spaces of toric varieties \cite{Ish04}. Consider the group
  $G = \GLr \times \GLs$, which acts on the space of matrices \M via change of
  basis. The rank of a matrix is the unique invariant for this action, the
  orbit closures being precisely the determinantal varieties \detvark. The
  assignments sending a variety $X$ to its arc space \arc X and its jet
  schemes \jetn X are functorial. Since $G$ is an algebraic group and its
  action on \M is rational, we see that \arc G and \jetn G are also groups,
  and that they act on \arc \M and \jetn \M, respectively.  Determinantal
  varieties are $G$-invariant, hence their arc spaces are $\arc G$-invariant
  and their jet schemes are $\jetn G$-invariant. The main observation is that
  most questions regarding components and dimensions of jet schemes and arc
  spaces of determinantal varieties can be reduced to the study of orbits in
  \arc \M and \jetn \M.

  Orbits in the arc space \arc \M are easy to classify. As a set, \arc \M is
  just the space of matrices with coefficients in \psr\CC, and \arc G acts via
  change of basis over the ring \psr\CC. Gaussian elimination allows us to
  find representatives for the orbits: each of them contains a unique diagonal
  matrix of the form $\operatorname{diag} (t^{\lambda_1}, \dots,
  t^{\lambda_r})$, where $\infty \ge \lambda_1 \ge \dots \ge \lambda_r \ge 0$,
  and the sequence $\lambda = (\lambda_1, \dots, \lambda_r)$ determines the
  orbit. In Section \ref{orbdec} we see how to decompose arc spaces and jet
  schemes of determinantal varieties as unions of these orbits. Once this is
  done, the main difficulty to determine irreducible components is the
  understanding of the generalized Nash problem for orbits closures in \arc\M.
  This is the purpose of the following theorem, which is proven in the article
  as Theorem \ref{orbposet:main}.
  
  \begin{introtheorem}[Nash problem for invariant valuations]
    Consider two sequences $\lambda = (\lambda_1 \ge \dots \ge \lambda_r \ge
    0)$ and $\lambda' = (\lambda'_1 \ge \dots \ge \lambda'_r \ge 0)$, and let
    \orb\lambda and \orb{\lambda'} be the corresponding orbits in the arc
    space \arc\M. Then the closure of \orb\lambda contains \orb{\lambda'} if
    and only if
    \[
      \lambda_r + \lambda_{r-1} + \dots + \lambda_{r-k}
      \quad
      \le
      \quad
      \lambda'_r + \lambda'_{r-1} + \dots + \lambda'_{r-k}
      \qquad
      \forall k \in \{0,\dots,r\}.  
    \]
  \end{introtheorem}

  Sequences of the form $(\infty \geq \lambda_1 \geq \lambda_2 \geq \dots \geq
  \lambda_r \geq 0)$ are closely related to partitions (the only difference
  being the possible presence of infinite terms), and the order that appears
  in the above theorem is a modification of a well-known order on the set of
  partitions: the order of domination (see Section \ref{sec:partitions}).
  Since the poset of partitions is well understood, one has very explicit
  information about the structure of the poset of orbits in the arc space.
  This allows us to compute minimal elements among some interesting families
  of orbits, leading to the proof of Theorem \ref{introthm1} (see Section
  \ref{orbposet}).

  \subsection{Log canonical thresholds}

  \Mustata's formula \cite{Mus01, ELM04, dFEI08} allows us to compute log
  discrepancies for divisorial valuations by computing codimensions of the
  appropriate sets in the arc space. In the case at hand, the most natural
  valuations one can look at are the invariant divisorial valuations. In
  Section \ref{discreps} we see that the maximal divisorial sets corresponding
  to these valuations are precisely the orbit closures in \M. Hence computing
  log discrepancies gets reduced to computing codimensions of orbits. This
  explains the relevance of the following result, which appears in Section
  \ref{discreps} as Proposition \ref{discreps:orbcodim}.

  \begin{introtheorem}[Log discrepancies of invariant valuations]
    Consider a sequence $\lambda = (\lambda_1 \ge \dots \ge \lambda_r \ge 0)$
    and let \orb\lambda be the corresponding orbit in the arc space \arc\M.
    Then the codimension of \orb\lambda in \arc\M is
    \[ 
      \operatorname{codim}(\orb\lambda, \arc\M) =
      \sum_{i=1}^r \lambda_i (s-r+2i-1).
    \]
  \end{introtheorem}

  Once these codimensions are known, one can compute log canonical thresholds for
  pairs involving determinantal varieties. The following result appears in
  Section \ref{discreps} as Theorem \ref{dicreps:lct}.

  \begin{introtheorem}
    Let \M be the space of matrices of size $r \times s$, and \detvark the
    subvariety of \M containing matrices of rank at most $k$. The log
    canonical threshold of the pair $(\M, \detvark)$ is
    \[
      \lct (\M, \detvark) =
      \min_{i=0 \dots k} \frac{(r-i)(s-i)}{k+1-i}.
    \]
  \end{introtheorem}

  We should note that the previous result is not new. Log resolutions for
  generic determinantal varieties are now classical objects. They are
  essentially spaces of complete collineations, obtained by blowing up
  \detvark along $\detvar0$, $\detvar1$, \dots, $\detvar{k-1}$, in this order
  \cite{Sem51,Tyr56,Vai84,Lak85}. It is possible to use these resolutions to
  compute log canonical thresholds, and this was done by A.\ Johnson in her
  Ph.D.\ thesis \cite{Joh03}. In fact she is able to compute all the multiplier
  ideals $\mathcal{J}(\M,c \cdot \detvark)$. Our method does not need any
  knowledge about the structure of these log resolutions.

  \subsection{Topological zeta function}

  Using our techniques, we are able to understand orbits in \arc\M quite
  explicitly. In Section \ref{motiv} we compute motivic volumes of orbits, and
  this allows us to determine topological zeta functions for determinantal
  varieties (for square matrices).

  \begin{introtheorem}
    Let $\M = \AA^{r^2}$ be the space of square $r \times r$ matrices, and let
    \detvark be the subvariety of matrices of rank at most $k$. Then the
    topological zeta function of the pair $(\M, \detvark)$ is given by
    \[
      Z^{\operatorname{top}}_{\detvark}(s) 
      =
      \prod_{\zeta \in \Omega} \frac{1}{1-s\,\zeta^{-1}}
    \]
    where $\Omega$ is the set of poles:
    \[
      \Omega = \left\{
      \,\,
      -\frac{r^2}{k+1},\quad
      -\frac{(r-1)^2}{k},\quad
      -\frac{(r-2)^2}{k-1},\quad
      \dots,\quad
      -(r-k)^2
      \,\,
      \right\}.
    \]
  \end{introtheorem}

  %\begin{introtheorem}
  %  Let $\M = \AA^{r^2}$ be the space of square $r \times r$ matrices, and let
  %  \detvark be the subvariety of matrices of rank at most $k$. Then the
  %  topological zeta function of the pair $(\M, \detvark)$ is given by
  %  \[
  %    Z^{\operatorname{top}}_{\detvark}(s) 
  %    =
  %    n_{r,k}
  %    \prod_{j=0}^{k}\frac{1}{(r-j)^2+s(k+1-j)},
  %  \]
  %  where $n_{r,k}$ is the integer given by:
  %  \[
  %    n_{r,k}
  %    =
  %    \frac{((r-2)!)^2((r-3)!)^{r-2}(r!)^2}{((r-k-1)!)^2}.
  %  \]
  %\end{introtheorem}

  \subsection*{Acknowledgments} 

  I am very grateful to my thesis adviser, Lawrence Ein, for introducing me to
  the world of arc spaces, and for suggesting this problem to me. His guidance
  and support during my years as a graduate student made this work possible. I
  would also like to thank Karen Smith and Mircea \Mustata for their valuable
  comments and suggestions on an earlier version of this paper.

\section{Arc spaces and motivic integration}

  \label{sec:introarcs}

  We briefly review in this section the basic theory of arc spaces and motivic
  integration, as these tools will be used repeatedly. Most of these results
  are well-known. We have gathered them mainly from \cite{DL98}, \cite{DL99},
  \cite{ELM04}, \cite{Ish06}, \cite{Vey06}, and \cite{dFEI08}.  We direct the
  reader to those papers for more details and proofs.

  We will always work with varieties and schemes defined over the complex
  numbers. 
  %Most of our results should remain true over an algebraically closed
  %field of arbitrary characteristic. 
  When we use the word scheme, we do not
  necessarily assume that it is of finite type. 

  \subsection{Arcs and jets}\label{sec-arcs} Given a variety $X$ and a
  non-negative integer $n$, we consider the following functor from the
  category of $\CC$-algebras to the category of sets:
  \[
     F_X^n(A) = \Hom\left(\Spec \tpsr n A, X\right).
  \]
  This functor is representable by a scheme, the \emph{$n$-th jet scheme}
  of $X$, which we denote by \jetn X. By construction, when $m \ge n$ we
  have natural projections $\psi_{m,n} : \jetm X \to \jetn X$, known as the
  \emph{truncation maps}. The inverse limit of the jet schemes of X (with
  respect to the truncation maps) is again a scheme, known as the \emph{arc
  space} of $X$, and denoted by \arc X:
  \[
      \arc X = \varprojlim_n \jetn X.
  \]
  Notice that \arc X is not of finite type if $\dim X >0$, and that we have
  natural projections $\psi_n: \arc X \to \jetn X$, also known as
  \emph{truncation maps}. When $X$ is affine, the arc space represents the
  following functor:
  \[
     F_X^\infty(A) = \Hom\left(\Spec \psr A, X\right).
  \]
  
  The assignment $X \mapsto \arc X$ is functorial: each morphism $f:X'\to X$
  induces by composition a morphism $\arc f: \arc X' \to \arc X$, and
  $\arc{(g\circ f)} = \arc f \circ \arc g$. As a consequence, if $G$ is a
  group scheme, so is \arc G, and if $X$ has an action by $G$, the arc space
  \arc X has an action by \arc G. Analogous statements hold for the jet
  schemes.

  \subsection{Contact loci and valuations}\label{sec-clv} By
  \emph{constructible} subset of a scheme (not necessarily Noetherian) we
  mean a finite union of locally closed subsets. A constructible subset $\C
  \subset \arc X$ is called \emph{thin} if one can find a proper subscheme
  $Y \subset X$ such that $\C \subset \arc Y$.  Constructible subsets which
  are not thin are called \emph{fat}. A \emph{cylinder} in \arc X is a set
  of the form $\psi^{-1}_n(C)$, for some constructible set $C\subset \jetn
  X$. On a smooth variety, cylinders are fat, but in general a cylinder
  might be contained in $\arc S$, where $S = \operatorname{Sing}(X) \subset
  X$ is the singular locus.

  An arc $\alpha \in \arc X$ induces a morphism $\alpha : \Spec \psr K \to X$,
  where $K$ is the residue field of $\alpha$. Given an ideal $\II \subset
  \OO_X$, its pull-back $\alpha^*(\II) \subset \psr K$ is of the form $(t^e)$,
  where $e$ is either a non-negative integer or infinity (by convention
  $t^\infty=0$). We call $e$ the \emph{order of contact} of $\alpha$ along
  $\II$ and denote it by $\ord_\alpha(\II)$.  Given a collection of ideals $I
  = (\II_1, \dots, \II_r)$ and a multi-index $\mu = (m_1, \dots, m_r) \in
  \ZZ_{\ge0}$ we introduce the \emph{contact locus}:
  \begin{gather*}
     \operatorname{Cont}^{=\mu}(I) = \{ \,\, \alpha \in \arc X \,\,\, : \,\,\, 
     \ord_\alpha(\II_j) = m_j \text{~for all $j$} \,\,\, \}, \\
     \operatorname{Cont}^{\mu}(I) = \{ \,\, \alpha \in \arc X \,\,\, : \,\,\, 
     \ord_\alpha(\II_j) \ge m_j \text{~for all $j$} \,\,\, \}.
  \end{gather*}
  Notice that contact loci are cylinders.

  Let $\C \subset \arc X$ be an irreducible fat set. Then $\C$ contains a
  generic point $\gamma \in \C$ which we interpret as a morphism $\gamma :
  \Spec \psr K \to X$, where $K$ is the residue field of $\gamma$. Let $\eta$
  be the generic point of $\Spec \psr K$.  Since $\C$ is fat, $\gamma(\eta)$
  is the generic point of $X$, and we get an inclusion of fields \[ \CC(X) \to
  K(\!(t)\!). \] The composition of this inclusion with the canonical
  valuation on $K(\!(t)\!)$ is a valuation on $\CC(X)$, which we denote by
  $\nu_\C$. In this way we obtain a map from the set of fat irreducible
  subsets of \arc X to the set of valuation of $\CC(X)$ defined over $X$:
  \[
    \def\sp{\,\,\,}
    \{\sp \C \subseteq \arc X \sp : \sp \C \sp\text{irreducible fat} \sp\}
    \quad\longrightarrow\quad
    \{\sp \text{discrete valuations over $X$} \sp\}.
  \]
  This map is always surjective: for a discrete valuation $\nu$ of $\CC(X)$
  defined over $X$, the completion of the discrete valuation ring $\OO_\nu$ is
  isomorphic to a power series ring $\psr{k_\nu}$. But it is far from being
  injective. For example, different choices of uniformizing parameter in a
  discrete valuation ring give rise to different arcs.

  A valuation $\nu$ of $\CC(X)$ is called \emph{divisorial} if it is of the
  form $q \cdot \operatorname{val}_E$, where $q$ is a positive integer and $E$
  is a prime divisor on a variety $X'$ birational to $X$. An irreducible fat
  set $\C \subset \arc X$ is said to be \emph{divisorial} if the corresponding
  valuation $\nu_\C$ is divisorial. In \cite{Ish06} it is shown that the union
  of all divisorial sets corresponding to a given valuation $\nu$ is itself a
  divisorial set defining $\nu$ (in fact it is an irreducible component of a
  contact locus). These unions are called \emph{maximal divisorial sets}.
  There is a one to one correspondence between divisorial valuations and
  maximal divisorial sets. This gives an inclusion
  \[
    \def\sp{\,\,\,}
    \{\sp \text{divisorial valuations over $X$} \sp\}
    \quad\hookrightarrow\quad
    \{\sp \C \subseteq \arc X \sp : \sp \C \sp\text{irreducible fat} \sp\}.
  \]
  Through this inclusion, the topology on the arc space $\arc X$ gives
  structure to the set of divisorial valuations. For example, given two
  valuations $\nu$ and $\nu'$ with corresponding maximal divisorial sets $\C$
  and $\C'$, we say that $\nu$ \emph{dominates} $\nu'$ if $\C \supseteq \C'$.
  The \emph{generalized Nash problem} consists in understanding the relation
  of domination among divisorial valuations.

  \subsection{Discrepancies}\label{sec-nashblup} Let $X$ be a variety of
  dimension $n$. The \emph{Nash blowing-up} of $X$, denoted $\widehat X$, is
  defined as the closure of $X_{\operatorname{reg}}$ in $\PP_X(\Omega^n_X)$;
  it is equipped with a tautological line bundle
  $\OO_{\PP_X(\Omega^n_X)}(1)|_{\widehat X}$, which we denote by $\widehat
  K_X$ and call the \emph{Mather canonical line bundle} of $X$. When $X$ is
  smooth, $X = \widehat X$ and $K_X = \widehat K_X$.

  When $Y$ is a smooth variety and $f:Y \to X$ is a birational morphism that
  factors through the Nash blowing-up, we define the \emph{relative Mather
  canonical divisor} of $f$ as the unique effective divisor supported on the
  exceptional locus of $f$ and linearly equivalent to $K_Y - f^*\widehat K_X$; we
  denote it by $\widehat K_{Y/X}$.

  Let $\nu$ be a divisorial valuation of $X$. Then we can find a smooth
  variety $Y$ and a birational map $Y \to X$ factoring through the Nash
  blowing-up of $X$, such that $\nu = q\cdot \operatorname{val}_E$ for some
  prime divisor $E \subset Y$. We define the \emph{Mather discrepancy} of $X$
  along $\nu$ as \[ \hat k_\nu(X) = q \cdot \ord_E(\widehat K_{Y/X}). \]
  This definition is independent of the choice of resolution $Y$.

  In the smooth case, \Mustata showed that we can compute discrepancies
  using the arc space \cite{Mus01,ELM04}. This is generalized to arbitrary
  varieties in \cite{dFEI08} via the use of Mather discrepancies. More
  precisely, given a divisorial valuation $\nu = q\cdot
  \operatorname{val}_E$, let $\C_\nu \subset \arc X$ be the corresponding
  maximal divisorial set.  Then \[ \operatorname{codim}(\C_\nu, \arc X) =
  \hat k_\nu(X) + q. \]

  \subsection{Motivic integration}\label{sec-motivic}

  {
  \def\M{\ensuremath{\mathcal{M}}\xspace}%
  \def\mm#1{\ensuremath{\mu_{#1}}\xspace}%

  Let $\M_0$ be the Grothendieck ring of algebraic varieties over
  $\CC$. In \cite{DL99}, the authors introduce a certain completion of
  a localization of $\M_0$, which we denote by $\M$.  Also, for each variety $X$
  over $\CC$, they define a measure \mm X on \arc X with values in \M. This
  measure is known as the \emph{motivic measure} of $X$. The following
  properties hold for $\M$ and the measures \mm X:
  \begin{itemize}

     \item There is a canonical ring homomorphism $\M_0 \to \M$. In particular,
       for each variety $X$ one can associate an element $[X] \in \M$, and the
       map $X \mapsto [X]$ is additive (meaning that $[X] = [Y] + [U]$, where $Y
       \subset X$ is a closed subvariety and $U=X\setminus Y$).

     \item The element $[\AA^1]\in \M$ has a multiplicative inverse. We write
       $\LL =[\AA^1]$.

     \item Both the Euler characteristic and the Hodge-Deligne polynomial,
       considered as ring homomorphisms with domain $\M_0$, extend to
       homomorphisms \[ \chi : \M \to \RR, \qquad E : \M \to \ZZ(\!(u,v)\!), \]
       where $\chi(\LL) = 1$ and $E(\LL) = uv$.

     \item Cylinders in \arc X are \mm X-measurable. In particular, contact
       loci are measurable.

     \item If $X$ is smooth, $\mm X(\arc X) = [X]$.

     \item A thin measurable set has measure zero.

     \item Let $\C \subset \arc X$ be a cylinder in \arc X. Then the
       truncations $\psi_n(\C) \subset X_n$ are of finite type, so they
       define elements $[\psi_n(\C)] \in \M$. Then
       \[ 
          \mm X (\C) = \lim_{n \to \infty} [\psi_n(\C)] \cdot \LL^{-nd} 
       \] 
       where $d$ is the dimension of $X$. Furthermore, if $\C$ does not
       intersect $\arc{(X_{\operatorname{sing}})}$, then
       $[\psi_n(\C)]\cdot\LL^{-nd}$ stabilizes for $n$ large enough.

     \item Given an ideal $\II \subset \OO_X$, we define a function $|\II|$
       on \arc X with values on \M\ via
       \[ 
          |\II|(\alpha) = \LL^{-\ord_\alpha(\II)} \qquad \alpha \in \arc X. 
       \] 
       Notice that $\ord_\beta(\II) = \infty$ if and only if $\beta \in
       \arc{\operatorname{Zeroes}(\II)}$, so $|\II|$ is only defined up to
       a measure zero set. Then $|\II|$ is \mm X-integrable and
       \[
          \int_{\arc X} |\II|\, d\mm X = \sum_{p=0}^\infty
          \mm X\left(\operatorname{Cont}^{=p}(\II)\right) \cdot \LL^{-p}.
       \]

     \item Let $f: Y \to X$ be a birational map, and assume $Y$ smooth. Let
       $\Jac(f)$ be the ideal in $\OO_Y$ for which $f^*\Omega_X^n \to
       \Jac(f) \cdot \Omega_Y^n$ is surjective.  Then $(\arc f)^*(\mm X) =
       |\Jac(f)| \cdot \mm Y$; in other words, for a measurable set $\C
       \subset \arc X$, and a \mm X-integrable function $\varphi$,
       \[
         \int_\C \varphi \,\, d\mm X = \int_{\arc f^{-1}(\C)}
         (\varphi \circ \arc f) \,\, |\Jac(f)| \,\, d\mm Y.
       \]
       This is known as the \emph{change of variables formula} for motivic
       integration.

     \item Assume that $X$ is smooth, and consider a subscheme $Y\subset X$
       with ideal $\II \subset \OO_X$. The \emph{motivic Igusa zeta
       function} of the pair $(X,Y)$ is defined as
       \[
          Z_Y(s) = \int_{\arc X} |\II|^s\, d\mm X = \sum_{p=0}^\infty
         \mm X\left(\operatorname{Cont}^{=p}(\II)\right) \cdot \LL^{-sp}.
       \]
       In this expression, $\LL^{-s}$ is to be understood as a formal
       variable, so $Z_Y(s) \in \M[\![\LL^{-s}]\!]$. It is shown in
       \cite{DL98} that $Z_Y(s)$ is a rational function. More precisely,
       let $f:X' \to X$ be a log resolution of the pair $(X,Y)$. This means
       that $f$ is a proper birational map, $X'$ is smooth, the
       scheme-theoretic inverse image of $Y$ in $X'$ is an effective
       Cartier divisor $E$, the map $f$ is an isomorphism on $X' \setminus
       E$, and the divisor $K_{X'/X}+E$ has simple normal crossings. Let
       $E_j$, $j \in J$, be the irreducible components of $E$, and write $E
       = \sum_{j\in J} a_j E_j$ and $K_{X'/X} = \sum_{j\in J} k_j E_j$. For
       a subset $I \subset J$, consider $E_I^\circ = (\cap_{i \in I} E_i)
       \setminus (\cup_{j \notin I} E_j)$. Then:
       \[
          Z_Y(s) = \sum_{I \subseteq J} [E_I^\circ]
          \prod_{i \in I} \frac { \LL^{-a_i s} ~ (\LL-1) } 
          { ~~ \LL^{k_i + 1} - \LL^{-a_i s} ~~ }.
       \]
       
     \item Keeping the same notation as above, the \emph{topological zeta
       function} of the pair $(X,Y)$ is defined as
       \[
         Z^{\operatorname{top}}_Y(s) =
         \sum_{I \subseteq J} \chi\left(E_I^\circ\right)
         \prod_{i \in I} \frac 1 {a_i s + k_i +1}
       \]
       where $\chi(\cdot)$ denotes the Euler characteristic (see
       \cite{Vey06, DL98}). Notice that when $s$ is a positive integer we
       have
       \[
          \frac { \LL^{-a_i s} ~ (\LL-1) } 
          { ~~ \LL^{k_i + 1} - \LL^{-a_i s} ~~ }
          =
          \frac 1 { 1 + \LL + \LL^2 + \dots + \LL^{a_i s + k_i} }
          =
          \frac 1 {[\PP^{a_i s + k_i}]},
       \]
       and that the Euler characteristic of $\PP^d$ is $d+1$. In this
       sense, the topological zeta function can be understood as the
       specialization of the motivic Igusa zeta function via the Euler
       characteristic map $\chi:\M \to \RR$. One can use this fact to show
       that $Z^{\operatorname{top}}_Y(s)$ is independent of the chosen log
       resolution \cite{DL98}.

     \item As explained in \cite[\S 2.3]{DL98} one can also understand the
       relation between $Z_Y(s)$ and $Z^{\operatorname{top}}_Y(s)$ in the
       following way. Let $\mathcal A$ be the subring of
       $\M[\![\LL^{-s}]\!]$ generated by the polynomials in $\M[\LL^{-s}]$
       and by the quotients of the form $\frac {\LL^{-as} (\LL-1)}
       {\LL^{k + 1} - \LL^{-as}}$ for positive integers $a, k$. The
       expansions
       \begin{align*}
         & \LL^{-s} = \sum_{n=0}^\infty \binom {-s} n (\LL - 1)^n
         = 1 - s (\LL -1) + \frac {-s (-s-1)} 2 (\LL -1)^2 + \dots
         \\\\ &
         \frac { \LL^{-a s} ~ (\LL-1) } 
         { \LL^{k + 1} - \LL^{-a s} }
         =
         \frac 1 { 1 + \LL + \LL^2 + \dots + \LL^{a s + k} }
         =
         \frac 1 {a s + k + 1} - \frac {a s + k}{2(a s + k + 1)}(\LL-1) +
         \dots
       \end{align*}
       induce a map from $\mathcal A$ to $\bar\M(s)[\![\LL-1]\!]$, where
       $\bar\M$ is the largest quotient of $\M$ with no $(\LL-1)$-torsion.
       Using the Euler characteristic map $\chi: \M \to \RR$ (which sends
       $\LL$ to $1$ and factors through $\bar\M$) and considering the
       quotient by the ideal generated by $(\LL-1)$, we get a natural map
       $\operatorname{ev}_{\LL=1} : \mathcal A \to \RR(s)$. The motivic
       Igusa zeta function is an element of $\mathcal A$, and the
       topological zeta function is its image in $\RR(s)$ via
       $\operatorname{ev}_{\LL=1}$:
       \[
          Z^{\operatorname{top}}_Y(s) = \operatorname{ev}_{\LL=1}
          \big( Z_Y(s) \big).
       \]
     
       % The motivic Igusa zeta function specializes to the topological
       % zeta function in the following way. Formally expanding $\LL^{-s}$
       % and the denominators in $Z_Y(s)$ as power series in $(\LL-1)$, one
       % gets a well defined element $\tilde Z_Y(s) \in
       % \M(s)[\![\LL-1]\!]$. Using the Euler characteristic map $\chi:\M
       % \to \RR$ and considering the quotient by the ideal generated by
       % $\LL-1$ we obtain an element $Z^{\operatorname{top}}_Y(s) \in
       % \RR(s)$. The rational function $Z^{\operatorname{top}}_Y(s)$ is
       % known as the \emph{topological zeta function} for the pair
       % $(X,Y)$.

  \end{itemize}

  }

\section{Partitions}

  \label{sec:partitions}

  In order to enumerate orbits in the arc space of determinantal varieties, it
  will be convenient to use the language of partitions.  In fact, we will
  consider a slight generalization of the concept of partition, where we allow
  terms of infinite size and an infinite number of terms (we call these
  objects pre-partitions).  In this section we recall some basic facts about
  partitions that will be needed in the rest of the article, and extend them
  to the case of pre-partitions. Most of the results are well known. For a
  detailed account of the theory of partitions we refer the reader to
  \cite{dCEP80} and \cite{Ful97}.

  \subsection{Definitions} Let \NN denote the set of non-negative integers,
  and consider $\NNb = \NN \cup \{\infty\}$. We extend the natural order on
  \NN to \NNb by setting $\infty > n$ for any $n \in \NN$.
  We also set $\infty + n = \infty$ for any $n \in \NNb$.

  A \emph{pre-partition} is an infinite non-increasing sequence of elements of
  \NNb. Given a pre-partition $\lambda = (\lambda_1, \lambda_2, \dots)$, the
  elements $\lambda_i$ are known as the \emph{terms} of $\lambda$. The first
  term $\lambda_1$ is called the \emph{maximal term} or \emph{co-length} of
  $\lambda$. If all the terms of $\lambda$ are non-zero, we say that $\lambda$
  has infinite length; otherwise, the largest integer $i$ such that $\lambda_i
  \ne 0$ is called the \emph{length} of $\lambda$. If a pre-partition
  $\lambda$ has length no bigger than $\ell$, we will often denote $\lambda$
  by the finite sequence $(\lambda_1, \lambda_2, \dots, \lambda_\ell)$.

  A \emph{partition} is a finite non-increasing sequence of positive integers.
  A partition can be naturally identified with a pre-partition of finite
  length and finite co-length.

  Given a pre-partition $\lambda = (\lambda_1, \lambda_2, \dots)$ we define
  \[
     \lambda_i^* = 
     \sup \left\{\, j \,\,:\,\, \lambda_j \ge i \,\right\} 
     \in \NNb.
  \]
  Then $\lambda_i^* \ge \lambda_{i+1}^*$, and we obtain a new pre-partition
  $\lambda^*$, known as the \emph{conjugate} pre-partition of $\lambda$. It
  follows from the definition that $\lambda^{**} = \lambda$, that the length
  of $\lambda^*$ is the co-length of $\lambda$, and that the co-length of
  $\lambda^*$ is the length of $\lambda$.  In particular, the conjugate of a
  partition is also a partition.

  \subsection{Diagrams} It will be helpful to visualize pre-partitions as
  Young diagrams (sometimes also known as Ferrers diagrams). A Young diagram
  is a graphical representation of a pre-partition; it is a collection of
  boxes, arranged in left-justified rows, with non-increasing row sizes. To
  each pre-partition $\lambda = (\lambda_1, \lambda_2, \dots)$ there is a
  unique Young diagram whose $i$-th row has size $\lambda_i$. For example:
  \diagramexample
  
  The length of a partition corresponds to the height of the associated
  diagram, whereas the co-length corresponds to the width. The diagram of the
  conjugate pre-partition is obtained from the original diagram by switching
  rows with columns. More concretely, if $T$ denotes the diagram associated to
  a pre-partition $\lambda$, the terms $\lambda_i$ of the pre-partition give
  the row sizes of $T$, and the terms $\lambda_i^*$ of the conjugate
  pre-partition give the column sizes of $T$.

  \subsection{Posets of partitions} Given two pre-partitions $\lambda =
  (\lambda_1, \lambda_2, \dots)$ and $\mu = (\mu_1, \mu_2, \dots)$, we say
  that $\lambda$ is \emph{contained} in $\mu$, and denote it by $\lambda
  \subseteq \mu$, if $\lambda_i \leq \mu_i$ for all $i$. Containment of
  pre-partitions corresponds to containment of the associated diagrams. In
  particular, $\lambda \subseteq \mu$ if and only if $\lambda^* \subseteq
  \mu^*$.

  If $\lambda$ and $\mu$ are pre-partitions with finite co-length, we say that
  $\mu$ \emph{dominates} $\lambda$, denoted by $\lambda \leq \mu$, if
  \[
     \lambda_1 + \lambda_2 + \dots + \lambda_i
     \quad \leq \quad
     \mu_1 + \mu_2 + \dots + \mu_i
  \]
  for all positive integers $i$. If $\lambda$ and $\mu$ have finite length, we
  say that $\mu$ \emph{co-dominates} $\lambda$, denoted $\lambda \triangleleft
  \mu$, if
  \[
     \lambda_i + \lambda_{i+1} + \dots
     \quad \leq \quad
     \mu_i + \mu_{i+1} + \dots
  \]
  for all positive integers $i$ (notice that the sums above have only a finite
  number of terms because the pre-partitions have finite length). It is shown
  in \cite[Prop.  1.1]{dCEP80} that the conditions of domination and
  co-domination of pre-partitions can be expressed in terms of the conjugates.
  More precisely, we have:
  \[ 
     \begin{array}{ccccrclc}
       \lambda \leq \mu &
       \quad \Longleftrightarrow \quad &
       \lambda^* \triangleleft \mu^* &
       \quad \Longleftrightarrow \quad &
       \lambda_i^* + \lambda_{i+1}^* + \dots &
       \leq &
       \mu_i^* + \mu_{i+1}^* + \dots &
       \quad \forall i,
       \\\\
       \lambda \triangleleft \mu &
       \quad \Longleftrightarrow \quad &
       \lambda^* \leq \mu^* &
       \quad \Longleftrightarrow \quad &
       \lambda_1^* + \lambda_2^* + \dots + \lambda_i^* &
       \leq &
       \mu_1^* + \mu_2^* + \dots + \mu_i^* &
       \quad \forall i. 
     \end{array}
  \]

  The three relations (containment, domination and co-domination) define
  partial orders.  We are mostly interested in the order of co-domination.
  Given a positive integer $r$, we denote by \eptr (respectively \ptr) the
  poset of pre-partitions (resp.~partitions) of length at most $r$ with the
  order of co-domination. By \bptrn we denote the poset of partitions of
  length at most $r$ and co-length at most $n$. It can be shown that \eptr,
  \ptr and \bptrn are all latices.

  \subsection{Adjacencies}\label{adjacencies} In Section \ref{orbposet} we
  will need to have a good understanding of the structure of the posets \eptr.
  For our purposes, it will be enough to determine the adjacencies in \eptr.

  Let $\lambda$ and $\mu$ be two different pre-partitions in \eptr such that
  $\lambda \triangleleft \mu$.  We say that $\lambda$ and $\mu$ are
  \emph{adjacent} (or that $\mu$ \emph{covers} $\lambda$) if there is no
  pre-partition $\nu$ in \eptr, different from $\lambda$ and $\mu$, such that
  $\lambda \triangleleft \nu \triangleleft \mu$.  Adjacencies in \ptr were
  determined in \cite[Prop. 1.2]{dCEP80}. They come in three different types,
  which we call \emph{single removals}, \emph{slips} and \emph{falls}.

  \begin{itemize}

     \item We say that a pre-partition $\lambda$ is obtained from $\mu$ via a
       \emph{single removal} if $\lambda_i = \mu_i$ for all $i \ne j$ and
       $\lambda_j = \mu_j - 1$, where $j$ is the smallest integer such that
       $\mu_j$ is finite. At the level of diagrams, $\lambda$ is obtained from
       $\mu$ by removing one box in the lowest row of finite size. Notice that
       this removal can only be done if $\mu_{j+1} < \mu_j$.

     \item We say that $\lambda$ is obtained from $\mu$ via a \emph{slip} if
       there exists a positive integer $j$ such that $\lambda_{j+1} =
       \mu_{j+1} - 1$, $\lambda_j = \mu_j + 1$ and $\lambda_i = \mu_i$ for all
       $i \notin\{ j, j+1\}$. In this case, the diagram of $\lambda$ is
       obtained from the diagram of $\mu$ by moving a box from row $j+1$ to
       row $j$. A slip from row $j+1$ can only happen if $\mu_{j+2} <
       \mu_{j+1}$, and $\mu_j < \mu_{j-1}$.

     \item We say that $\lambda$ is obtained from $\mu$ via a \emph{fall} if
       $\mu^*$ is obtained from $\lambda^*$ via a slip. In other words, there
       exist positive integers $j < k$ such that $\lambda_k = \mu_k - 1$,
       $\lambda_j = \mu_j + 1$, and $\lambda_i = \mu_i$ for all $i \notin \{j,
       k\}$. A fall from row $k$ to row $j$ can only happen if $\mu_{k+1} <
       \mu_k$, $\mu_j < \mu_{j-1}$, and $\mu_i = \mu_{i'}$ for all $i, i' \in
       \{j, j+1, \dots, k\}$. During a fall, a box in the diagram of $\mu$ is
       moved from one column to the next.

  \end{itemize}
  Since we are dealing with pre-partitions, we will also need to consider
  \emph{infinite removals}:
  \begin{itemize}

     \item We say that a pre-partition $\lambda$ is obtained from $\mu$ via an
       \emph{infinite removal} if $\lambda_i = \mu_i$ for all $i \ne j$ and
       $\lambda_j < \mu_j = \infty$, where $j$ is the largest integer such
       that $\mu_j$ is infinite. At the level of diagrams, $\lambda$ is
       obtained from $\mu$ by removing infinitely many boxes in the highest
       row of infinite size.

  \end{itemize}

  \partitionsfigure

  In \cite{dCEP80} the authors show that adjacencies in the set of partitions
  with respect to the order of domination correspond to simple removals, slips
  and falls. The result for partitions with the order of co-domination follows
  immediately from the fact that $\lambda \triangleleft \mu \Leftrightarrow
  \lambda^* \leq \mu^*$. Now consider two pre-partitions $\lambda
  \triangleleft \mu$ with finite length, and assume they are adjacent. They
  must have the same number of infinite terms, otherwise the pre-partition
  $\nu$ obtained from $\lambda$ by adding one box in the lowest finite row
  verifies $\lambda \triangleleft \nu \triangleleft \mu$. Let $\hat\lambda$
  and $\hat\mu$ be the partitions obtained from $\lambda$ and $\mu$ by
  removing the infinite terms. Then $\hat\lambda$ and $\hat\mu$ are adjacent
  and we can apply the result of \cite{dCEP80} to show that $\lambda$ can be
  obtained from $\mu$ via a simple removal, a slip, or a fall.

  \begin{theorem}\label{th-adjacencies}
     Let $\lambda$ and $\mu$ be two pre-partitions in \eptr, and assume that
     $\lambda \triangleleft \mu$. Then there exists a finite sequence of
     pre-partitions in \eptr,
     \[
        \lambda = \nu^m 
        \,\,\triangleleft\,\, 
        \cdots
        \,\,\triangleleft\,\, 
        \nu^1 
        \,\,\triangleleft\,\,
        \nu^0 = \mu,
     \]
     such that $\nu^i$ is obtained from $\nu^{i-1}$ via a simple removal, an
     infinite removal, a slip, or a fall.
  \end{theorem}

  \begin{proof}
     Assume first that there are more infinite terms in $\mu$ that in
     $\lambda$. To each infinite row $j$ in $\mu$ which is finite in $\lambda$
     we apply an infinite removal, leaving at least $\lambda_j$ boxes
     (depending on the particular $\mu$ we might need to leave more boxes).
     This way we obtain a sequence
     \[
        \lambda
        \,\,\triangleleft\,\, 
        \nu^{m_0} 
        \,\,\triangleleft\,\, 
        \cdots
        \,\,\triangleleft\,\, 
        \nu^1 
        \,\,\triangleleft\,\,
        \nu^0 = \mu,
     \]
     where $\lambda$ has the same number of infinite rows as $\nu^{m_0}$ and
     where each $\nu^i$ is obtained from $\nu^{i-1}$ via an infinite removal.
     Notice that $m_0$ is the number of rows which are infinite in $\mu$ but
     finite in $\lambda$. Since $\mu$ has finite length, $m_0$ is finite.

     Let $k$ be the number of boxes in the finite rows of $\nu^{m_0}$. Then any
     pre-partition with the same number of infinite rows as $\nu^{m_0}$ and
     co-dominated by $\nu^{m_0}$ must use at most $k$ boxes in its finite rows.
     In particular there are only finitely many such pre-partitions. It follows
     that we can find a finite sequence
     \[
        \lambda =
        \nu^m 
        \,\,\triangleleft\,\, 
        \cdots
        \,\,\triangleleft\,\, 
        \nu^{m_0+1}
        \,\,\triangleleft\,\,
        \nu^{m_0}
     \]
     where consecutive terms are adjacent. From the discussion preceding the
     theorem, we see that $\nu^i$ can be obtained from $\nu^{i-1}$ by a simple
     removal, a slip, or a fall, and the result follows.
  \end{proof}

\section{Orbit decomposition of the arc space}

  \label{orbdec}
  We start by recalling our basic setup from the introduction. $\M = \AA^{rs}$
  is the space of $r \times s$ matrices with coefficients in \CC, and we
  assume that $r \leq s$. The ring of regular functions on \M is a polynomial
  ring on the entries of a generic matrix $x$:
  \[
    \OO_\M = \CC[x_{11},\dots,x_{rs}],
    \qquad \qquad
    x = 
    \begin{pmatrix}
      x_{11} & x_{12} & \dots  & x_{1s} \\
      x_{21} & x_{22} & \dots  & x_{2s} \\
      \vdots & \vdots & \ddots & \vdots \\
      x_{r1} & x_{r2} & \dots  & x_{rs}
    \end{pmatrix}.
  \]

  The generic determinantal variety of matrices of rank at most $k$ is
  denoted by \detvark. The ideal of \detvark is generated by the $(k+1)
  \times (k+1)$ minors of $x$.  It can be shown that generic determinantal
  varieties are irreducible, and that the singular locus of \detvark is
  $\detvar{k-1}$ when $0<k<r$ (\detvar0 and $\detvar{r}$ are smooth).  They
  are also Cohen-Macaulay, Gorenstein, and have rational singularities.
  For proofs of the previous statements, and a comprehensive account of the
  theory of determinantal varieties we refer the reader to \cite{BV88}.

  We denote by \G the group $\GLr \times \GLs$. It acts naturally on \M via
  change of basis:
  \[
    (g,h) \cdot A = g\,A\,h^{-1},
    \qquad
    (g,h) \in \G,
    \quad
    A \in \M.
  \]
  The rank of a matrix is the only invariant for this action, and the
  determinantal varieties are the orbit closures (their ideals being the only
  invariant prime ideals of $\OO_\M$).

  The group \G is an algebraic group. In particular it is an algebraic
  variety, and we can consider its arc space \arc\G and its jet schemes
  \jetn\G. The action of \G on \M and \detvark induces actions at the level
  of arc spaces and jet schemes:
  \[
    \begin{aligned}
      &\arc G \times \arc \M \to \arc \M,
      &\qquad
      &\arc G \times \arc \detvark \to \arc \detvark, \\
      &\jetn G \times \jetn \M \to \jetn \M,
      &&\jetn G \times \jetn \detvark \to \jetn \detvark.
    \end{aligned} 
  \]
  In this section we classify the orbits associated to all of these actions.

  As a set, the arc space \arc \M contains matrices of size $r \times s$ with
  entries in the power series ring \psr\CC. Analogously, the group $\arc G =
  \parc \GLr \times \parc \GLs$ is formed by pairs of square matrices with
  entries in \psr\CC which are invertible, that is, their determinant is a
  unit in \psr\CC. Orbits in \arc \M correspond to similarity classes of
  matrices over the ring \psr\CC, and we can easily classify these using the
  fact that \psr\CC is a principal ideal domain. 

  \begin{definition}[Orbit associated to a partition]\label{orbdec:deforb}
    Let $\lambda = (\lambda_1, \lambda_2, \dots, \lambda_r) \in \ept r$ be a
    pre-partition with length at most $r$. Consider the following matrix in
    \arc M:
    \[
      \delta_\lambda 
      \quad =  \quad
      \def\cd{\cdots}%
      \def\vd{\vdots}%
      \def\dd{\ddots}%
      \def\t#1{t^{\lambda_#1}}%
      \begin{pmatrix}
        0   & \cd & 0    & \t1 & 0   & \cd & 0   \\ 
        0   & \cd & 0    & 0   & \t2 & \cd & 0   \\ 
        \vd &     & \vd  & \vd & \vd & \dd & \vd \\ 
        0   & \cd & 0    & 0   & 0   & \cd & \t r 
      \end{pmatrix}
    \]
    (we use the convention that $t^\infty = 0$). Then the \arc\G-orbit of
    the matrix $\delta_\lambda$ is called the orbit in \arc\M
    \emph{associated} to the pre-partition $\lambda$, and it is denoted by
    \orb\lambda.
  \end{definition}

  \begin{proposition}[Orbits in \arc \M]
    \label{orbdec:orbdec}
    
    Every \arc G-orbit of \arc \M is of the form \orb \lambda for some
    pre-partition $\lambda \in \ept r$. An orbit \orb \lambda is contained in
    \arc \detvark if and only if the associated pre-partition $\lambda$
    contains at least $r-k$ leading infinities, i.e.  $\lambda_1 = \dots =
    \lambda_{r-k} = \infty$. In particular, $\arc \M \setminus
    \arc{\detvar{r-1}}$ is the union of the orbits corresponding to
    partitions, and the orbits in $\arc\detvark \setminus
    \arc{\detvar{k-1}}$ are in bijection with $\pt{k}$.  Moreover, the
    orbit corresponding to the empty partition $(0, 0, \dots )$ is the arc
    space $\parc{\M \setminus \detvar{r-1}}$.
  \end{proposition}

  \begin{proof}
    As mentioned above, \arc M is the set of $r \times s$-matrices with
    coefficients in the ring \psr\CC, and the group \arc G acts on \arc M via row
    and column operations, also with coefficients in \psr\CC. Using Gaussian
    elimination and the fact that \psr\CC is a principal ideal domain, we see
    that each \arc G-orbit in \arc M contains a matrix of the form
    \[
      \def\cd{\cdots}%
      \def\vd{\vdots}%
      \def\dd{\ddots}%
      \def\t#1{t^{\lambda_#1}}%
      \begin{pmatrix}
        0   & \cd & 0    & \t1 & 0   & \cd & 0   \\ 
        0   & \cd & 0    & 0   & \t2 & \cd & 0   \\ 
        \vd &     & \vd  & \vd & \vd & \dd & \vd \\ 
        0   & \cd & 0    & 0   & 0   & \cd & \t r 
      \end{pmatrix},
    \]
    where the exponents $\lambda_i$ are non-negative integers or $\infty$,
    and we use the convention that $t^\infty = 0$. After permuting columns
    and rows, we can assume that the exponents of these powers form a
    non-increasing sequence:
    \[
        \lambda_1 \geq \lambda_2 \geq \dots \geq \lambda_r.
    \]
    Moreover, the usual structure theorems for finitely generated modules
    over principal ideal domains guarantee that each orbit contains a
    unique matrix in this form. This shows that the set of \arc G-orbits
    in \arc M is in bijection with \eptr.

    The ideal defining \detvark in \M is generated by the minors of size
    $(k+1) \times (k+1)$. Let $\lambda \in \ept r$ be a pre-partition of
    length at most $r$ and consider $\delta_\lambda$ as in
    Definition~\ref{orbdec:deforb}.  The $(k+1)\times(k+1)$ minors of
    $\delta_\lambda$ are either zero or of the form $\prod_{i \in I}
    t^{\lambda_i}$, where $I$ is a subset of $\{1,\dots,r\}$ with $k+1$
    elements. For all of the minors to be zero, we need at least $r-k$
    infinities in the set $\{\lambda_1, \dots, \lambda_r\}$ (recall that
    $t^\infty = 0$). In other words, $\delta_\lambda$ is contained in $\arc
    \detvark$ if and only if $\lambda$ contains $r-k$ leading infinities.

    The variety \detvark is invariant under the action of $G$, so \arc\detvark
    is invariant under the action of \arc G. In particular the orbit \orb
    \lambda is contained in \arc\detvark if and only if $\delta_\lambda$ is.
    The rest of the proposition follows immediately.
  \end{proof}

  \begin{proposition}[Orbits and contact loci]
    \label{orbdec:cont}
    The contact locus \contp \detvark is invariant under the action of \arc G,
    and the orbits contained in \contp \detvark correspond to those
    pre-partitions $\lambda \in \ept r$ whose last $k+1$ terms add up to at
    least $p$:
    \[
      \lambda_{r-k} + \dots + \lambda_r \ge p.
    \]
  \end{proposition}

  \begin{proof}
    The truncations maps from the arc space to the jet schemes are in fact
    natural transformations of functors. This means that we have the following
    natural diagram:
    \[
      \xymatrix @C=0.01in {
        \arc G \ar[d] &\times& \arc M \ar[d] \ar[rrrr] & & & & \arc M \ar[d] \\
        \jetn G &\times& \jetn M \ar[rrrr] & & & & \jetn M
      }
    \]
    Since \detvark is $G$-invariant, \jetn\detvark is \jetn G-invariant, so
    \contp\detvark (the inverse image of $\jet{p-1}\detvark$ under the
    truncation map) is \arc G-invariant. In particular, an orbit \orb \lambda
    is contained in \contp\detvark if and only if its base point
    $\delta_\lambda$ is. The order of vanishing of $\II_{\detvark}$ along
    $\delta_\lambda$ is $\lambda_{r-k} + \dots + \lambda_r$ (recall that
    $\II_{\detvark}$ is generated by the minors of size $(k+1)\times(k+1)$ and
    that $\lambda_1 \ge \dots \ge \lambda_r$). Hence $\delta_\lambda$ belongs
    to \contp\detvark if and only if $\lambda_{r-k} + \dots + \lambda_r \ge
    p$, and the proposition follows.
  \end{proof}

  \begin{proposition}[Orbits are cylinders]
    \label{orbdec:cyl}
    Let $\lambda \in \ept r$ be an pre-partition, and let \orb\lambda be the
    associated orbit in \arc \M. If $\lambda$ is a partition, \orb\lambda is a
    cylinder of \arc\M. More generally, let $r-k$ be the number of infinite
    terms of $\lambda$. Then \orb\lambda is a cylinder of \arc\detvark.
  \end{proposition}

  \begin{proof}
    Assume that $\lambda$ is an pre-partition with $r-k$ leading infinities,
    and consider the following cylinders in \arc \M:
    \begin{align*}
      A_\lambda &=
        \cont{\lambda_r}{\detvar0} \cap
        \cont{\lambda_r+\lambda_{r-1}}{\detvar1} \cap
        \cont{\lambda_r+\dots+\lambda_{r-k}}{\detvar {k}}, \\
      B_\lambda &=
        \cont{\lambda_r+1}{\detvar0} \cup
        \cont{\lambda_r+\lambda_{r-1}+1}{\detvar1} \cup
        \cont{\lambda_r+\dots+\lambda_{r-k}+1}{\detvar {k}}.
    \end{align*}
    By Propositions \ref{orbdec:orbdec} and \ref{orbdec:cont}, we know that
    \[
      \orb\lambda = (A_\lambda \setminus B_\lambda) \cap \arc\detvark.
    \]
    Hence \orb \lambda is a cylinder in \arc\detvark, as required.
  \end{proof}

  \begin{remark}
    The proof of Proposition \ref{orbdec:cyl} tells us that we can express all
    orbits in \arc\M with contact conditions with respect to the determinantal
    varieties. In particular, if we know the order of contact of an arc with
    respect to all determinantal varieties, we know which orbit it belongs to.
    Moreover, this is also true not only for closed points, but for all
    schematic points of \arc\M. It follows that every point of \arc\M (closed
    or not) is contained in an orbit generated by a closed point.
  \end{remark}

  We now study the jet schemes \jetn G and \jetn \M. As in the case of the arc
  space, elements in \jetn G and \jetn \M are given by matrices, but now the
  coefficients lie in the ring $\tpsr n\CC$.

  \begin{definition}
    Let $\lambda = (\lambda_1, \dots, \lambda_\ell) \in \bpt r{n+1}$ be a
    partition with length at most $r$ and co-length at most $n+1$. Then the
    matrix $\delta_\lambda$ considered in Definition \ref{orbdec:deforb}
    gives an element of the jet scheme \jetn\M. The \jetn G-orbit of
    $\delta_\lambda$ is called the orbit of \jetn\M \emph{associated} to
    $\lambda$ and it is denoted by $\torb \lambda n$.
  \end{definition}

  \begin{proposition}[Orbits in \jetn \M]
    Every \jetn G-orbit of \jetn \M is of the form $\torb\lambda n$ for some
    partition $\lambda \in \bpt r{n+1}$. An orbit $\torb\lambda n$ is contained
    in \jetn \detvark if and only if the associated partition contains at
    least $r-k$ terms equal to $n+1$. In particular, the set of orbits in
    $\jetn\detvark \setminus \jetn{\detvar{k-1}}$ is in bijection with \bpt
    kn.
  \end{proposition}

  \begin{proof}
    This can be proven in the same way as Proposition \ref{orbdec:orbdec}. The
    only difference is that we now work with a principal ideal ring $\tpsr
    n\CC$, as opposed to with the principal ideal domain \psr\CC, but the
    domain condition played no role in the proof of Proposition
    \ref{orbdec:orbdec}.  Alternatively, one can notice that $\tpsr n\CC$ is a
    quotient of \psr\CC, so modules over $\tpsr n\CC$ correspond to modules
    over \psr\CC with the appropriate annihilator, and one can reduce the
    problem of classifying \jetn G-orbits in \jetn \M to classifying \arc
    G-orbits in \arc \M with bounded exponents.
  \end{proof}

  \begin{definition}
    Let $\lambda = (\lambda_1, \lambda_2, \dots)$ be a pre-partition,
    and let $n$ be a nonnegative integer. Then the \emph{truncation} of $\lambda$
    to level $n$ is the partition $\overline\lambda = (\overline\lambda_1,
    \overline\lambda_2, \dots)$ where
    \[
      \overline\lambda_i = \min(\lambda_i, n).
    \]
  \end{definition}

  \begin{proposition}[Truncation of orbits]
    \label{orbdec:trunc}
    
    Let $\lambda \in \ept r$ be a pre-partition, and let $\overline\lambda$ be
    its truncation to level $n+1$. Then the image of \orb\lambda under the
    natural truncation map $\arc \M \to \jetn \M$ is
    $\torb{\overline\lambda}n$.  Conversely, fix a partition $\overline\lambda
    \in \bpt r{n+1}$, and let $\Gamma \subset \ept r$ be the set of
    pre-partitions whose truncation to level $n+1$ is $\overline\lambda$. Then
    the inverse image of $\torb{\overline\lambda}n$ under the truncation map
    is the union of the orbits of \arc\M corresponding to the pre-partitions
    in $\Gamma$.
  \end{proposition}

  \begin{proof}
    Notice that $\delta_{\overline\lambda} \in \jetn\M$ is the truncation of
    $\delta_\lambda \in \arc\M$. Then the fact that the truncation of
    $\orb\lambda = \arc G \cdot \delta_\lambda$ equals
    $\torb{\overline\lambda}n = \jetn G \cdot \delta_{\overline\lambda}$ is an
    immediate consequence of the fact that the truncation map is a natural
    transformation of functors (see the proof of Proposition
    \ref{orbdec:cont}). Conversely, if $\lambda$ and $\lambda'$ have different
    truncations, the \jetn G-orbits $\torb{\overline\lambda} n$ and
    $\torb{\overline\lambda'}n$ are different, so \orb{\lambda'} is not in the
    fiber of $\torb{\overline\lambda}n$.
  \end{proof}

\section{Orbit poset and irreducible components of jet schemes}

  \label{orbposet}
  After obtaining a classification of the orbits of the action of $\arc \G =
  \parc \GLr \times \parc \GLs$ on \arc \M and \arc \detvark, we start the
  study of their geometry. The first basic question is the following: how are
  these orbits positioned with respect to each other inside the arc space \arc
  \M? We can make this precise by introducing the notion of orbit poset.

  \begin{definition}[Orbit poset]
    \def\ca{\orb{}} \def\cb{\ensuremath{\orb{}'}\xspace}%
    
    Let \ca and \cb be two \arc G-orbits in \arc \M. We say that \ca
    \emph{dominates} \cb, and denote it by $\cb \le \ca$, if \cb is contained
    in the Zariski closure of \ca. The relation of dominance defines a partial
    order on the set of \arc G-orbits of \arc \M.  The pair $(\arc \M / \arc
    G, \,\le)$ is known as the \emph{orbit poset} of \arc \M. 
  \end{definition}

   Our goal is to prove that the bijection that maps a pre-partition to its
   associated orbit in \arc\M is in fact an order-reversing isomorphism
   between \ptr and the orbit poset.  At this stage it is not hard to show
   that one of the directions of this bijection reverses the order.

  \begin{proposition}[Domination of orbits implies co-domination of partitions]
    \label{orbposet:domimpsub}
    Let $\lambda, \lambda' \in \ept r$ be two pre-partitions of length at
    most $r$, and let \orb \lambda and \orb{\lambda'} be the associated orbits
    in \arc \M. If \orb\lambda dominates \orb{\lambda'}, then $\lambda'$
    co-dominates $\lambda$:
    \[
      \orb \lambda \ge \orb{\lambda'}
      \quad
      \Longrightarrow
      \quad
      \lambda \triangleleft \lambda'.
    \]
  \end{proposition}

  \begin{proof}
    From Proposition \ref{orbdec:cont} we get that $\orb \lambda \subset
    \contp\detvark$, where $p=\lambda_r + \dots + \lambda_{r-k}$.  Since a
    contact locus is always Zariski closed, if \orb \lambda dominates
    \orb{\lambda'} we also know that $\orb{\lambda'} \subset \contp\detvark$.
    Again by Proposition \ref{orbdec:cont}, this gives $\lambda'_r + \dots +
    \lambda'_{r-k} \ge p$, as required.
  \end{proof}

  We now proceed to prove the converse to Proposition
  \ref{orbposet:domimpsub}. Given two pre-partitions $\lambda, \lambda' \in
  \ept r$ with $\lambda \triangleleft \lambda'$, we need to show that the
  closure of \orb\lambda contains \orb{\lambda'}. We will exhibit this
  containment by producing a ``path'' in the arc space \arc \M whose general
  point is in \orb\lambda but specializes to a point in \orb{\lambda'}. These
  types of ``paths'' are known as \emph{wedges}.

  \begin{definition}[Wedge]\label{wedge}
    Let $X$ be a scheme over \CC. A \emph{wedge} $w$ on $X$ is a morphism of
    schemes $w : \operatorname{Spec} \CC[\![s,t]\!] \to X$. Given a wedge $w$,
    one can consider the diagram
    \[
      \xymatrix{
        **[l] \Spec \psr\CC
        \ar@/_/[rd]^{s\,\mapsto\,0} \ar@/^2pc/[rrrrd]^{w_0}
        \\& **[r] \Spec \CC[\![s,t]\!] \ar[rrr]^(.65)w
        & & & X .
        \\ **[l] \Spec \CC(\!(s)\!)[\![t]\!] 
        \ar@/^/[ru] \ar@/_2pc/[rrrru]_{w_s}
      }
    \]
    The map $w_0$ is known as the \emph{special arc} of $w$, and $w_s$ as the
    \emph{generic arc} of $w$.
  \end{definition}

  \begin{lemma}
    \label{lem:wedge}
    Let $w$ be a wedge on \M and let \orb0 be the \arc\G-orbit in \arc\M of
    the special arc $w_0$ of $w$. Assume that there is a \arc\G-orbit $\orb s$
    in \arc\M containing the generic arc $w_s$ of $w$.  Then $\orb s$
    dominates \orb0.
  \end{lemma}

  \begin{proof}
    By the hypothesis, the closure of $\orb s$ contains $w_0$ (because $w_0$
    is in the closure of $w_s$). But the closure of an orbit is invariant, so
    $\overline{\orb s}$ must contain $\orb 0 = \arc\G \cdot w_0$.
  \end{proof}

  \begin{lemma}
    \label{lem:removals}
    Let $\lambda$ and $\mu$ be two pre-partitions in \eptr and assume that
    $\lambda$ is obtained from $\mu$ via a removal (simple or infinite).
    Then \orb\lambda dominates \orb\mu.
  \end{lemma}

  \begin{proof}
    Let $i$ be the index such that $\lambda_i < \mu_i$, and consider the following
    wedge on \M:
    \[
      w =
      \renewcommand{\arraystretch}{0.4}
      \left(\,\,
      \begin{matrix}
               0         \\
        &      \ddots    \\
        &&     0         \\
        &&&    s t^{\lambda_i} + t^{\mu_i} \\
        &&&&   t^{\lambda_{i+1}} \\
        &&&&&  \ddots    \\
        &&&&&& t^{\lambda_r}
      \end{matrix}
      \,\,
      \right).
    \]
    The special arc of $w$ is $\delta_\mu$. The generic arc $w_s$ only differs
    from $\delta_\lambda$ by the presence of the unit $s+t^{\mu_i-\lambda_i}$
    on row $i$. Therefore $w_s$ and $\delta_\lambda$ have the same contact
    with respect to all the determinantal varieties, and the proof of
    Proposition \ref{orbdec:cyl} shows that $w_s$ is contained in \orb\lambda.
    Now we can apply Lemma \ref{lem:wedge} with $\orb0 = \orb\mu$ and $\orb s
    = \orb \lambda$, and the result follows.
  \end{proof}

  \begin{lemma}
    \label{lem:falls}
    Let $\lambda$ and $\mu$ be two pre-partitions in \eptr and assume that
    $\lambda$ is obtained from $\mu$ via a slip or a fall.  Then \orb\lambda
    dominates \orb\mu.
  \end{lemma}

  \begin{proof}
    Let $i<j$ be the indices such that $\lambda_i = \mu_i+1$,
    $\lambda_j = \mu_j-1$ and $\mu_k = \lambda_k$ for $k\ne i,j$.
    Consider the following wedge:
    \[
      w = \\ \left(
      \renewcommand{\arraystretch}{0.4}
      \begin{array}{cccccccccccccccccccccc}
                    t^{\lambda_1} \\
        &           \ddots  \\
        &&          t^{\lambda_{i-1}} \\
        &&&         \alpha & 0 & \cdots & 0 & \beta \\
        &&&         0 & t^{\lambda_{i+1}} & \cdots & 0 & 0 \\
        &&&         \vdots & \vdots & \ddots & \vdots & \vdots \\
        &&&         0 & 0 & \cdots & t^{\lambda_{j-1}} & 0 \\
        &&&         \gamma & 0 & \cdots & 0 & \delta \\
        &&&&&&&&    t^{\lambda_{j+1}} \\
        &&&&&&&&&   \ddots \\
        &&&&&&&&&&  t^{\lambda_r} \\
      \end{array}
      \right)
    \]
    where
    \[
     \begin{pmatrix}
       \alpha & \beta \\
       \gamma & \delta 
     \end{pmatrix}
     =
     \begin{pmatrix}
       s t^{\lambda_i} + t^{\lambda_i-1} & t^{\lambda_i-1} \\
       s t^{\lambda_j} & s t^{\lambda_j} + t^{\lambda_j+1}
     \end{pmatrix}.
    \]
    Notice that
    \[
      \operatorname{ord}_t \begin{pmatrix} \alpha & \beta \\ \gamma & \delta \end{pmatrix}
      = \lambda_j, 
      \quad \qquad
      \operatorname{ord}_t \begin{pmatrix} \alpha & \beta \\ \gamma & \delta \end{pmatrix}
        \Bigg|_{s=0}
      = \lambda_j+1, 
      \quad \qquad
      \det \begin{pmatrix} \alpha & \beta \\ \gamma & \delta \end{pmatrix}
        = t^{\lambda_i + \lambda_j} \left( 1 + s t + s^2 \right).
    \]
    From these equations, we see that $w_0$ and $\delta_\mu$ have the same
    order of contact with respect to all determinantal varieties, and the
    proof of Proposition \ref{orbdec:cyl} tells us that $w_0$ is contained in
    \orb\mu. Analogously, the equations above show that $w_s$ is contained in
    \orb\lambda. Lemma \ref{lem:wedge} gives the result.
  \end{proof}

  \begin{theorem}[Orbit poset $=$ Pre-partition poset]\label{orbposet:main}
    The map that sends a pre-partition $\lambda \in \ept r$ of length at most
    $r$ to the associated orbit \orb \lambda in \arc \M is an order-reversing
    isomorphism between \eptr and the orbit poset:
    %\[
    %  (\arc \M / \arc G, \le)^{\operatorname{op}}
    %  \quad \simeq \quad
    %  (\ept r, \triangleleft)
    %\]
    \[
      \orb\lambda \ge \orb{\mu}
      \quad \Longleftrightarrow \quad
      \lambda \triangleleft \mu
      \quad \Longleftrightarrow \quad
      \lambda_{r-i} + \dots + \lambda_r
      \le
      \mu_{r-i} + \dots + \mu_r
      \quad
      \forall i.
    \]
  \end{theorem}

  \begin{proof}
    The theorem follows from Proposition \ref{orbposet:domimpsub}, Lemma
    \ref{lem:removals}, Lemma \ref{lem:falls}, and Theorem
    \ref{th-adjacencies}.
  \end{proof}

  In the remainder of the section we use Theorem \ref{orbposet:main} to
  compute the number of irreducible components of the jet schemes of generic
  determinantal varieties.

  \begin{notation}
    As it is customary in the theory of partitions, we write $\lambda =
    (d_1^{a_1} \dots d_j^{a_j})$ to denote the pre-partition that has $a_i$ copies
    of $d_i$. For example $(\infty,\infty,5,3,3,3,2,1,1) = (\infty^2\, 5^1\, 3^3\,
    2^1\,1^2)$.
  \end{notation}

  \begin{proposition}
    \label{jetcomp:cont}\def\c{\orb{}}%
    Recall that $\detvark \subset \M$ denotes the determinantal variety of
    matrices of size $r \times s$ and rank at most $k$.  Assume that
    $0<k<r-1$, and let \c be an irreducible component of $\contp\detvark
    \subset \arc \M$.  Then \c contains a unique dense \arc G-orbit
    \orb\lambda. Moreover, $\lambda$ is a partition (contains no infinite
    terms) and $\lambda = (d^{a+r-k}\, e^1)$ where
    \[
      p = (a+1)\, d + e,
      \qquad \qquad
      0 \le e < d,
    \]
    and either $e=0$ and $0 \le a\le k$ or $e>0$ and $0 \le a < k$.
    Conversely, for any partition as above, its associated orbit is dense in
    an irreducible component of \contp\detvark.
  \end{proposition}

  \begin{example}
    When $r=8$, $k=6$, and $p=5$, the partitions given by the proposition are
    \[
      (5,5), \quad (4,4,1), \quad (3,3,2), \quad (2,2,2,1), \quad (1,1,1,1,1,1).
    \]
    When $r=5$, $k=3$, and $p=5$, we only get
    \[
      (5,5), \quad (4,4,1), \quad (3,3,2), \quad (2,2,2,1).
    \]
  \end{example}

  \begin{proof}
    By Theorem \ref{orbposet:main} and Proposition \ref{orbdec:cont},
    computing the irreducible components of \contp\detvark is equivalent to
    computing the minimal elements (with respect to the order of
    co-domination) among all pre-partitions $\lambda \in \ept r$ such that
    \[
      \lambda_r + \lambda_{r-1} + \dots + \lambda_{r-k} \ge p.
    \]
    Let $\Sigma$ be the set of such partitions.  To find minimal elements in
    $\Sigma$ it will be useful to keep in mind the structure of the
    adjacencies in \eptr discussed in Section \ref{adjacencies}

    First notice that all minimal elements in $\Sigma$ must be partitions.
    Indeed, given an element $\lambda \in \Sigma$, truncating all infinite
    terms of $\lambda$ to a high enough number produces another element of
    $\Sigma$.  Moreover, if $\lambda \in \Sigma$ is minimal, we must have
    $\lambda_1 = \lambda_2 = \dots = \lambda_{r-k}$. If this were not the
    case, we could consider the partition $\lambda'$ such that $\lambda_1' =
    \dots = \lambda'_{r-k} = \lambda_{r-k}$, and $\lambda_i' = \lambda_i$ for
    $i>r-k$.  Then $\lambda'$ would also be in $\Sigma$, but $\lambda'
    \triangleleft \lambda$, contradicting the fact that $\lambda$ is minimal.
    It is also clear that minimal elements of $\Sigma$ must verify
    $\lambda_{r-k} + \dots + \lambda_r = p$. In fact, if a partition in
    $\Sigma$ does not verify this, we can decrease the last terms of the
    partition an still remain in $\Sigma$.

    So far we know that the minimal elements in $\Sigma$ are partitions that
    verify $\lambda_1 = \dots = \lambda_{r-k}$ and $\lambda_{r-k} + \dots +
    \lambda_r = p$. Note that we assume $0<k<r-1$, so for any two partitions
    $\lambda$, $\lambda'$ with the previous properties, if $\lambda_{r-k} \ne
    \lambda'_{r-k}$, then $\lambda$ and $\lambda'$ are not comparable.
    
    Pick a minimal element $\lambda \in \Sigma$, and write $d =
    \lambda_{r-k}$.  Let $\ell$ be the length of $\lambda$. The proposition
    will follow if we show that the sequence $(\lambda_{r-k}, \lambda_{r-k+1},
    \dots ,\lambda_\ell)$ is of the form $(d, \dots, d, e)$ for some $0\le
    e<d$. But this is clear from the analysis of the adjacencies in \eptr
    given in Section \ref{adjacencies}.  Consider the Young diagram $\Gamma$
    associated to $\lambda$. The longest row of $\Gamma$ has length $d$. If
    there are two rows, say $i<j$, with length less than $d$, then we must
    have $r-k < i$ and we can move one box from row $j$ to row $i$ (via a
    sequence of falls and slips) and obtain a partition still in $\Sigma$ but
    co-dominated by $\lambda$. This contradicts the fact that $\lambda$ is
    minimal, and we see that $\lambda$ must have the form given in the
    proposition.
  \end{proof}

  \begin{proposition}
    \label{jetcomp:conteasy}
    Assume that $k=0$ or $k=r-1$. Then $\contp\detvark \subset \arc \M$ is
    irreducible and contains a unique dense orbit \orb\lambda, where $\lambda
    = (p^{r-k})$.
  \end{proposition}

  \begin{proof}
    Form Proposition \ref{orbdec:cont}, the orbits \orb\lambda contained in
    \contp{\detvar0} are the ones that verify $\lambda_r \ge p$. It is clear
    that the minimal partition of this type is $(p^{r})$. Analogously,
    $\contp{\detvar{r-1}}$ contains orbits whose associated partitions verify
    $\lambda_1 + \dots + \lambda_r \ge p$, and the minimal one among these is
    $(p^1)$.
  \end{proof}

  \begin{theorem}
    \label{jetcomp:num}

    If $k=0$ or $k=r-1$, the contact locus $\contp\detvark \subset \arc \M$ is
    irreducible. Otherwise, the number of irreducible components of
    $\contp\detvark \subset \arc \M$ is
    \[
      p + 1 - \left\lceil \frac{p}{k+1} \right\rceil.
    \]
  \end{theorem}

  \begin{proof}
    The first assertion follows directly from Proposition
    \ref{jetcomp:conteasy}. For the second one, we need to count the number of
    partitions that appear in Proposition \ref{jetcomp:cont}. Recall that
    these were partitions of the form $\lambda^d = (d^{a+r-k}, e^1)$ of length
    at most $r$ such that $p = (a+1)d+e$ and $0\le e < d$. Since $d$ ranges
    from $0$ to $p$, we have at most $p+1$ such partitions. But as we decrease
    $d$, the length of $\lambda^d$ increases, possibly surpassing the limit
    $r$.  Therefore the number of allowed partitions is $p+1 - d_0$, where
    $d_0$ is the smallest integer such that $\lambda^{d_0}$ has length no
    greater than $r$.

    If $d$ divides $p$, the length of $\lambda^d$ is $(\tfrac{p}{d}-1+r-k)$.
    Otherwise it is $(\left\lfloor\tfrac{p}{d}\right\rfloor+r-k)$.  In either
    case, the length is no greater that $r$ if and only if $d\ge\left\lceil
    \tfrac{p}{k+1}\right\rceil$. Hence $d_0 =
    \left\lceil\tfrac{p}{k+1}\right\rceil$, and the theorem follows.
  \end{proof}

  \begin{corollary}\label{irrcomp:main}
    It $k=0$ or $k=r-1$, the jet scheme \jetn\detvark is irreducible.
    Otherwise, the number of irreducible components of \jetn\detvark is
    \[
      n + 2 - \left\lceil \frac{n+1}{k+1} \right\rceil.
    \]
  \end{corollary}

  \begin{proof}
    The contact locus $\cont{n+1}\detvark$ is the inverse image of the jet
    scheme $\jetn\detvark$ under the truncation map $\arc\M \to \jetn\M$.
    Since \M is smooth, this truncation map is surjective, so \jetn\detvark
    has the same number of components as $\cont{n+1}\detvark$. Now the result
    follows directly from Theorem \ref{jetcomp:num}.
  \end{proof}

\section{Discrepancies and log canonical thresholds}

  \label{discreps}

  In this section we compute discrepancies for all invariant divisorial
  valuations over \M and over \detvark, and use it to give formulas for log
  canonical thresholds involving determinantal varieties. We start with a
  proposition that determines all possible invariant maximal divisorial sets
  in terms of orbits in the arc space.

  \begin{proposition}[Divisorial sets = Orbit closures, \M]
    \label{discreps:diveqorb}%
    \def\c{\orb{}}%

    Let $\nu$ be a $G$-invariant divisorial valuation over \M, and let \c be
    the associated maximal divisorial set in \arc \M. Then there exists a
    unique partition $\lambda \in \pt r$ of length at most $r$ whose
    associated orbit \orb\lambda is dense in \c. Conversely, the closure of
    \orb\lambda, where $\lambda$ is a partition, is a maximal divisorial set
    associated to an invariant valuation.
    
  \end{proposition}

  \begin{proof}
    \def\c{\orb{}}%

    Recall from Section \ref{sec-clv} (or see \cite{Ish06}) that $\c$ is the
    union of the fat sets of \arc \M that induce the valuation $\nu$.
    Therefore, since $\nu$ is $G$-invariant, \c is \arc G-invariant and can be
    written as a union of orbits. Note that the thin orbits of \arc\M are all
    contained in $\arc{\detvar{r-1}}$, and that \c is itself fat, so \c must
    contain a fat orbit. Let $\Sigma \subset \pt r$ be the set of partitions
    indexing fat orbits contained in \c. For $\mu \in \Sigma$ we denote by
    $\nu_\mu$ the valuation induced by \orb\mu. Then, for $f \in \OO_\M$ we
    have:
    \[
      \nu(f) 
      \quad = \quad
      \min_{\gamma \in \c} \,\, \{ \, \ord_\gamma(f) \, \}
      \quad = \quad
      \min_{\mu \in \Sigma} \,\, \min_{\gamma \in \orb\mu} 
        \,\, \{ \, \ord_\gamma(f) \, \}
      \quad = \quad
      \min_{\mu \in \Sigma} \,\, \{ \, \nu_\mu(f) \, \}.
    \]
    As a consequence, since $\nu_\mu$ is determined by its value on the ideals
    $\II_{\detvar0}$,\dots,$\II_{\detvar{r-1}}$, the same property holds for
    $\nu$.  Let $\lambda = (\lambda_1,\dots,\lambda_r)$ be such that
    $\nu(\II_{\detvark}) = \lambda_r + \dots + \lambda_{r-k}$. From the fact
    that $\II_{\detvark} \, \II_{\detvar{k-2}} \subset \II_{\detvar{k-1}}^2$
    we deduce that $\lambda_k \ge \lambda_{k+1}$, and we get a partition
    $\lambda \in \pt r$ whose associated orbit \orb\lambda induces the
    valuation $\nu$ (so $\lambda \in \Sigma$). The proposition follows if we
    show that \orb\lambda is dense in \c.

    Consider $\mu \in \Sigma$. Since $\orb\mu \subset \c$, we know that
    $\nu_\mu \ge \nu$, and we get that
    \[
      \mu_r + \dots + \mu_{r-k} 
      \,\, = \,\,
      \nu_\mu(\II_{\detvark})
      \,\, \ge \,\,
      \nu(\II_{\detvark})
      \,\, = \,\, \lambda_r + \dots + \lambda_{r-k}.
    \]
    Hence $\lambda \triangleleft \mu$, and Theorem \ref{orbposet:main} tells us that
    \orb\mu is contained in the closure of \orb\lambda, as required.
  \end{proof}

  \begin{proposition}[Divisorial sets = Orbit closures, \detvark]
    \def\c{\orb{}}%
    
    Given a partition
    $\lambda = (\lambda_1, \dots, \lambda_\ell) \in \pt k$ of length at most
    $k$, denote by $\lambda^+ = (\infty, \dots, \infty, \lambda_1, \dots,
    \lambda_\ell) \in \ept r$ the pre-partition obtained by adjoining $r-k$
    infinities.
    Let $\nu$ be a $G$-invariant divisorial valuation over \detvark, and let
    \c be the associated maximal divisorial set in \arc\detvark. Then there
    exists a unique partition $\lambda \in \pt k$ such that the orbit
    $\orb{\lambda^+}$ is dense in \c.  Conversely, the closure of
    $\orb{\lambda^+}$, where $\lambda \in \pt k$, is a maximal divisorial set
    in \arc\detvark associated to a $G$-invariant divisorial valuation.
  \end{proposition}

  \begin{proof}
    Analogous to the proof of \ref{discreps:diveqorb}.
  \end{proof}

  We now proceed to compute discrepancies for invariant divisorial valuations.
  These are closely related to the codimensions of the corresponding maximal
  divisorial sets, which by the previous propositions are just given by orbit
  closures. Since orbits are cylinders, their codimension can be computed by
  looking at the corresponding orbit in a high enough jet scheme. But jet
  schemes are of finite type, so orbits have a finite dimension that can be
  computed via the codimension of the corresponding stabilizer. For this
  reason, we will try to understand the structure of the different stabilizers
  in the jet schemes \jetn G. 

  Recall from Definition \ref{orbdec:deforb} that \orb\lambda is the orbit
  containing the following matrix:
  \[
    \delta_\lambda = \begin{pmatrix}
      0      & \cdots & 0       & t^{\lambda_1} & 0              & \cdots & 0            \\ 
      0      & \cdots & 0       & 0             & t^{\lambda_2}  & \cdots & 0            \\ 
      \vdots &        & \vdots  & \vdots        & \vdots         & \ddots & \vdots       \\ 
      0      & \cdots & 0       & 0             & 0              & \cdots & t^{\lambda_r}  
    \end{pmatrix}.
  \]
  This matrix defines an element of the jet scheme \jetn\M as long as $n$ is
  greater than the co-length of $\lambda$; the corresponding \jetn G-orbit in
  \jetn\M is denoted by $\torb\lambda n$. The following proposition determines
  the codimension of the stabilizer of $\delta_\lambda$ in the jet group \jetn
  G.

  \begin{proposition}
    \label{discreps:stabcodim}
    \def\sta{\ensuremath{H_{\lambda,n}}\xspace}%

    Let $\lambda\in\pt r$ be a partition of length at most $r$, and let $n$ be
    a positive integer greater than the highest term of $\lambda$. Let \sta
    denote the stabilizer of $\delta_\lambda$ in the group \jetn G.  Then
    \[
      \operatorname{codim}(\sta, \jetn G) = 
      (n+1) r s - \sum_{i=1}^r \lambda_i(s-r +2i -1).
    \]
  \end{proposition}

  \begin{proof}
    \def\sta{\ensuremath{H_{\lambda,n}}\xspace}%

    Pick $(g,h) \in \jetn G = \jetn{(\GLr)} \times \jetn{(\GLs)}$. Then:
    \begin{gather*}
      (g,h) \in \sta
      \quad \Longleftrightarrow \quad
      g \cdot \delta_\lambda \cdot h^{-1} = \delta_\lambda
      \quad \Longleftrightarrow \quad
      g \cdot \delta_\lambda = \delta_\lambda \cdot h
      \quad \Longleftrightarrow \quad \\ \\
      \begin{pmatrix}
        0      & \cdots & 0       & t^{\lambda_1} \ast & t^{\lambda_2} \ast & \cdots & t^{\lambda_r} \ast\\ 
        0      & \cdots & 0       & t^{\lambda_1} \ast & t^{\lambda_2} \ast & \cdots & t^{\lambda_r} \ast\\ 
        \vdots &        & \vdots  & \vdots             & \vdots             & \ddots & \vdots            \\ 
        0      & \cdots & 0       & t^{\lambda_1} \ast & t^{\lambda_2} \ast & \cdots & t^{\lambda_r} \ast  
      \end{pmatrix}
      =
      \begin{pmatrix}
        t^{\lambda_1} \ast & \cdots & t^{\lambda_1} \ast & t^{\lambda_1} \ast & t^{\lambda_1} \ast & \cdots & t^{\lambda_1} \ast \\
        t^{\lambda_2} \ast & \cdots & t^{\lambda_2} \ast & t^{\lambda_2} \ast & t^{\lambda_2} \ast & \cdots & t^{\lambda_2} \ast \\
        \vdots &        & \vdots  & \vdots        & \vdots         & \ddots & \vdots       \\ 
        t^{\lambda_r} \ast & \cdots & t^{\lambda_r} \ast & t^{\lambda_r} \ast & t^{\lambda_r} \ast & \cdots & t^{\lambda_r} \ast \\
      \end{pmatrix}.
    \end{gather*}
    This equality of matrices gives one equation of the form $t^{a(i,j)}\ast =
    t^{b(i,j)}\ast$ for each entry $(i,j)$ in a $r \times s$ matrix. We have
    $a(i,j)=\lambda_{j-s+r}$ and $b(i,j)=\lambda_i$ (assume $\lambda_j = \infty$
    for $j<0$).
    
    Each equation of the form $t^a\ast = t^b\ast$ gives $(n+1)-\min\{a,b\}$
    independent equations on the coefficients of the power series, so it reduces
    the dimension of the stabilizer by $(n+1)-\min\{a,b\}$.  The entries $(i,j)$
    for which $\min\{a(i,j), b(i,j)\} = \lambda_k$ form an
    \parbox{1em}{~}\turnbox{90}{L}-shaped region of the $r\times s$ matrix, as
    we illustrate in the following diagram:
    \[
    %\left(\,\,
    \raisebox{-6.5ex}
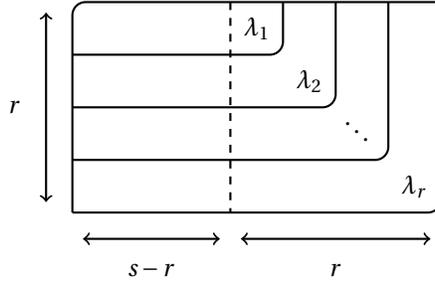
{
    \begin{tikzpicture}[scale=.7, thick]
      \draw[thick,rounded corners=5pt] (1,-1) -- (5,-1) -- (5,0);
      \draw[thick,rounded corners=5pt] (1,-2) -- (6,-2) -- (6,0);
      \draw[thick,rounded corners=5pt] (1,-3) -- (7,-3) -- (7,0);
      \draw[thick,rounded corners=5pt] (1,-4) -- (8,-4) -- (8,0);
      \draw[thick,rounded corners=5pt] (1,-4) -- (1,0) -- (8,0);
      \draw[dashed] (4,0) -- (4,-4);
      \draw (4.5,-.5) node {$\lambda_1$};
      \draw (5.5,-1.5) node {$\lambda_2$};
      \draw (6.4,-2.3) node {$\ddots$};
      \draw (7.5,-3.5) node {$\lambda_r$};
      \draw[<->] (.5,-.2) -- (.5,-3.8);
      \draw[<->] (1.2,-4.5) -- (3.8,-4.5);
      \draw[<->] (4.2,-4.5) -- (7.8,-4.5);
      \draw (-.1,-2) node {$r$};
      \draw (2.5,-5.1) node {$s-r$};
      \draw (6,-5.1) node {$r$};
    \end{tikzpicture}
    }
    %\,\,\right)
    \]
    The region corresponding to $\lambda_i$ contains $(s-r+2i-1)$ entries, and
    the result follows.
  \end{proof}

  \begin{proposition}
    \label{discreps:orbcodim}

    Let $\lambda\in\ept r$ be a pre-partition of length at most $r$, and
    consider its associated \arc\G-orbit \orb\lambda in \arc\M. If $\lambda$
    contains infinite terms, \orb\lambda has infinite codimension. If
    $\lambda$ is a partition, the codimension is given by:
    \[
      \operatorname{codim} (\orb\lambda, \arc\M) = 
      \sum_{i=1}^r \lambda_i(s-r +2i -1).
    \]
  \end{proposition}

  \begin{proof}
    If $\lambda$ contains infinite terms, \orb\lambda is thin, so it has
    infinite codimension. Otherwise Proposition \ref{orbdec:trunc} tells us
    that \orb \lambda is the inverse image of $\torb\lambda n$ under the
    truncation map $\arc \M \to \jetn \M$ for $n$ large enough. Since \M is
    smooth, we see that the codimension of \orb\lambda in \arc\M is the same
    as the codimension of $\torb\lambda n$ in \jetn\M. The dimension of
    $\torb\lambda n$ is the codimension of the stabilizer of $\delta_\lambda$
    in \jetn G. The result now follows from Proposition
    \ref{discreps:stabcodim} and the fact that \jetn\M has dimension
    $(n+1)rs$.
  \end{proof}

  \begin{corollary}
    Let $\nu$ be a $G$-invariant divisorial valuation of \M, and write $\nu
    = q \cdot \operatorname{val}_E$, where $q$ is a positive integer and
    $E$ is a divisor in a smooth birational model above $M$.
    Let $\lambda\in\pt r$ be
    the unique partition such that \orb\lambda induces $\nu$, and let \dis\nu\M be
    the discrepancy of \M along $\nu$. Then
    \[
      \dis\nu\M + q = 
      \sum_{i=1}^r \lambda_i(s-r +2i -1).
    \]
  \end{corollary}

  \begin{proof}
    From Proposition \ref{discreps:diveqorb} we know that the closure of
    \orb\lambda is the maximal divisorial set associated to $\nu$. Since \M is
    smooth, the log discrepancy $\dis\nu\M + q$ agrees with the
    codimension of the associated maximal divisorial set (see Section
    \ref{sec-nashblup}).  The result now follows from Proposition
    \ref{discreps:orbcodim}.
  \end{proof}

  \begin{theorem}
    \label{dicreps:lct}

    Recall that \M denotes the space of matrices of size $r \times s$ and
    \detvark is the variety of matrices of rank at most $k$. The log canonical
    threshold of the pair $(\M, \detvark)$ is
    \[
      \lct (\M, \detvark) =
      \min_{i=0 \dots k} \frac{(r-i)(s-i)}{k+1-i}.
    \]
  \end{theorem}

  \begin{proof}
    We will use \Mustata's formula (see \cite[Cor.~3.2]{ELM04}) to compute log
    canonical thresholds:
    \[
      \lct(\M, \detvark) = 
      \min_n \left\{ \frac{\operatorname{codim}(\jetn \detvark, \jetn \M)}{n+1} \right\}
      =
      \min_p \left\{ \frac{\operatorname{codim}(\contp \detvark, \arc \M)}{p} \right\}.
    \]
    Let $\Sigma_p \subset \ept r$ be the set of pre-partitions of length at
    most $r$ such that $\lambda_r + \dots + \lambda_{r-k} = p$.  By
    Propositions \ref{orbdec:cont} and \ref{jetcomp:cont}, we have:
    \[
      \lct(\M, \detvark) = 
      \min_p 
      \min_{\lambda \in \Sigma_p}
      \left\{ \frac{\operatorname{codim}(\orb \lambda, \arc \M)}{p} \right\}.
    \]
    Consider the following linear function
    \[
      \psi(a_1,\dots,a_r) = 
      \sum_{i=1}^r a_i(s-r +2i -1).
    \]
    Then, by Proposition \ref{discreps:orbcodim}, we get:
    \[
      \lct(\M, \detvark) = 
      \min_p 
      \min_{\lambda \in \Sigma_p}
      \left\{ \frac{\psi(\lambda)}{p} \right\}
      =
      \min_p 
      \min_{\lambda \in \Sigma_p}
      \left\{ \psi\left(\tfrac{\lambda}{p}\right) \right\}.
    \]
    Let $\Sigma \subset \QQ^r$ be the set of tuples $(a_1,\dots,a_r)$ such that
    $a_1 \ge a_2 \ge \dots \ge a_r \ge 0$ and $a_r + \dots + a_{r-k} = 1$. Then:
    \[
      \lct(\M, \detvark) = 
      \min_{a \in \Sigma}
      \left\{ \psi(a) \right\}.
    \]
    The map $\varphi(a_1, \dots ,a_r) = (a_1-a_2, \dots, a_{r-1}-a_r, a_r)$
    sends $\Sigma$ to $\Sigma'$, where $\Sigma' \subset \QQ^r$ is the set of
    tuples $(b_1,\dots,b_r)$ such that $b_i \ge 0$ and $(k+1) b_r + k b_{r-1}
    \dots + b_{r-k} = 1$. Then
    \[
      \lct(\M, \detvark) = 
      \min_{b \in \Sigma'}
      \left\{ \xi(b) \right\},
    \]
    where
    \[
      \xi(b) = \psi(\varphi^{-1}(b)) =
      \sum_{i=1}^r (b_r +b_{r-1}+\dots+b_i)(s-r +2i -1)
      =
      \sum_{j=1}^r b_j \, j \, (s-r +j).
    \]
    Note that in the definition of $\Sigma'$ the only restriction on the first
    $r-k-1$ coordinates $b_1, b_2, \dots, b_{r-k-1}$ is that they are
    nonnegative. 
    Let $\Sigma''$ be the subset of $\Sigma'$ obtained by setting $b_1 = \dots =
    b_{r-k-1} = 0$.
    From the formula for $\xi(b)$ we see that the minimum
    $\min_{b\in\Sigma'}\{\xi(b)\}$ must be achieved in $\Sigma''$. But
    $\Sigma''$ is a simplex and $\xi$ is linear, so the minimum is actually
    achieved in one of the extremal points of $\Sigma''$. These extremal points
    are:
    \begin{multline*}
      P_{r-k} = (0,\dots,0,1,0,\dots,0,0), \quad
      P_{r-k+1} = (0,\dots,0,0,\tfrac{1}{2},\dots,0,0),\quad
      \dots \\ \dots \quad
      P_{r-1} = (0,\dots,0,0,0,\dots,\tfrac{1}{k},0), \quad
      P_{r} = (0,\dots,0,0,0,\dots,0,\tfrac{1}{k+1}).
    \end{multline*}
    The value of $\xi$ at these points is:
    \[
      \xi(P_{r-i}) = \frac{1}{k+1-i}(r-i)(s-i).
    \]
    Therefore
    \[
      \lct (\M, \detvark) =
      \min_{i=0 \dots k} \frac{(r-i)(s-i)}{k+1-i},
    \]
    as required.
  \end{proof}

\section{Motivic integration}

  \label{motiv}

  In the previous section we computed codimensions of orbits in the arc space
  \arc \M, as a mean to obtain formulas for discrepancies and log canonical
  thresholds. But a careful look at the proofs shows that we can understand
  more about the orbits than just their codimensions. As an example of this,
  in this section we compute the motivic volume of the orbits in the arc
  space. This allows us to determine topological zeta functions of
  determinantal varieties.

  Throughout this section, we will restrict ourselves to the case of square
  matrices, i.e. we assume $r=s$.

  \subsection{Motivic volume of orbits}\label{sec:orbvolume}
  
  Before we state the main proposition, we need to recall some notions from
  the group theory of \GLr: parabolic subgroups, Levi factors, flag manifolds,
  and the natural way to obtain a parabolic subgroup from a partition.

  \begin{definition}
    Let $0 < v_1 < v_2 < \dots < v_j < r$ be integers. A \emph{flag} in
    $\CC^r$ of signature $(v_1, \dots, v_j)$ is a nested chain $V_1 \subset
    V_2 \subset \dots \subset V_j \subset \CC^r$ of vector subspaces with
    $\dim V_i = v_i$.  The general linear group \GLr acts transitively on the
    set of all flags with a given signature. The stabilizer of a flag is known
    as a \emph{parabolic subgroup} of \GLr. If $P \subset \GLr$ is a parabolic
    subgroup, the quotient $\GLr/P$ parametrizes flags of a given signature
    and it is known as a \emph{flag variety}.
  \end{definition}

  \begin{definition}
    Let $\{e_1, \dots, e_r\}$ be the standard basis for $\CC^r$, and let
    $\lambda = (d_1^{a_1} \dots d_j^{a_j}) \in \pt r$ be a partition. Write
    $a_{j+1} = r - \sum_{i=1}^j a_i$ and $v_i = a_1 + \dots + a_i$, and
    consider the following vector subspaces of $\CC^r$:
    \[
      V_i = \operatorname{span}(e_1, \dots , e_{v_i}),
      \qquad
      \qquad
      W_i = \operatorname{span}(e_{v_{i-1}+1}, \dots , e_{v_i}).
    \]
    We denote by $P_\lambda$ the stabilizer of the flag $V_1 \subset \dots
    \subset V_j$ and call it the \emph{parabolic subgroup of \GLr associated
    to $\lambda$}. The group $L_\lambda = \GL{a_1} \times \dots \times
    \GL{a_{j+1}}$ embeds naturally in $P_\lambda$ as the group endomorphisms
    of $W_i$, and it is known as the \emph{Levi factor} of the parabolic
    $P_\lambda$.
  \end{definition}

  \begin{example}
    Assume $r=6$ and consider the partition $\lambda = (4,4,4,1,1) = (4^3
    1^2)$.  Then $P_\lambda$ and $L_\lambda$ are the groups of invertible
    $r\times r$ matrices of the forms
    \[
      P_\lambda :\,\,
      \begin{pmatrix}
        \ast & \ast & \ast & \ast & \ast & \ast \\
        \ast & \ast & \ast & \ast & \ast & \ast \\
        \ast & \ast & \ast & \ast & \ast & \ast \\
        0    & 0    & 0    & \ast & \ast & \ast \\
        0    & 0    & 0    & \ast & \ast & \ast \\
        0    & 0    & 0    & 0    & 0    & \ast \\
      \end{pmatrix},
      \qquad
      \qquad
      L_\lambda :\,\,
      \begin{pmatrix}
        \ast & \ast & \ast & 0    & 0    & 0    \\
        \ast & \ast & \ast & 0    & 0    & 0    \\
        \ast & \ast & \ast & 0    & 0    & 0    \\
        0    & 0    & 0    & \ast & \ast & 0    \\
        0    & 0    & 0    & \ast & \ast & 0    \\
        0    & 0    & 0    & 0    & 0    & \ast \\
      \end{pmatrix}.
    \]
  \end{example}

  \begin{proposition}
    \label{prop:orbvolume}
    Assume that $r=s$. Let $\lambda \in \pt r$ be a partition of length at
    most $r$ and consider its associated parabolic subgroup $P_\lambda$ and
    Levi factor $L_\lambda$.  Let $\mu$ be the motivic measure in \arc \M, and
    \orb \lambda the orbit in \arc \M associated to $\lambda$. If $b$ is the
    log discrepancy of the valuation induced by \orb\lambda, we have:
    \[
      \mu(\orb\lambda) \quad=\quad
      \LL^{-b} \, [\GLr/P_\lambda]^2 \, [L_\lambda].
    \]
  \end{proposition}

  \begin{proof}
    \def\sta{\ensuremath{H_{\lambda,n}}\xspace}%
    \def\g{\mathfrak{g}}%
    \def\h{\mathfrak{h}}%

    Consider $n$, $\delta_\lambda$ and $\sta \subset \jetn G$ as in Proposition
    \ref{discreps:stabcodim}. If \torb \lambda n is the truncation of \orb
    \lambda to \jetn \M, we know that for $n$ large enough
    \[
      \mu(\orb\lambda) \,\,=\,\, \LL^{-r^2n} \, [\torb\lambda n]
      \,\,=\,\, \LL^{-r^2n}\, [\jetn G] \, [\sta]^{-1}
      \,\,=\,\, \LL^{r^2n}\, [\GLr]^2 \, [\sta]^{-1}.
    \]
    At the beginning of the proof of Proposition \ref{discreps:stabcodim} we
    found the equations defining \sta:
    \begin{gather}
      (g,h) \in \sta
      \qquad \Leftrightarrow \qquad
      g \cdot \delta_\lambda \cdot h^{-1} = \delta_\lambda
      \qquad \Leftrightarrow \qquad
      g \cdot \delta_\lambda = \delta_\lambda \cdot h
      \qquad \Leftrightarrow \qquad \nonumber\\ \nonumber\\
      \label{mateq}
      \begin{pmatrix}
        t^{\lambda_1} \ast & t^{\lambda_2} \ast & \cdots & t^{\lambda_r} \ast\\ 
        t^{\lambda_1} \ast & t^{\lambda_2} \ast & \cdots & t^{\lambda_r} \ast\\ 
        \vdots             & \vdots             & \ddots & \vdots            \\ 
        t^{\lambda_1} \ast & t^{\lambda_2} \ast & \cdots & t^{\lambda_r} \ast  
      \end{pmatrix}
      =
      \begin{pmatrix}
        t^{\lambda_1} \ast & t^{\lambda_1} \ast & \cdots & t^{\lambda_1} \ast \\
        t^{\lambda_2} \ast & t^{\lambda_2} \ast & \cdots & t^{\lambda_2} \ast \\
        \vdots             & \vdots             & \ddots & \vdots       \\ 
        t^{\lambda_r} \ast & t^{\lambda_r} \ast & \cdots & t^{\lambda_r} \ast \\
      \end{pmatrix}.%\tag{$\ast$}
    \end{gather}

    As a variety, \jetn G can be written as product $G \times \g^{2n}$, where
    $\g \simeq \AA^{r^2}$ is the Lie algebra of $\GLr$. Let $g_{i,j}^{(k)}$
    and $h_{i,j}^{(k)}$ be the natural coordinates on $\jetn G =
    \left(\GLr\times\g\right)^2$, where $g_{i,j}^{(0)} = g_{i,j}$ and
    $h_{i,j}^{(0)} = h_{i,j}$ are coordinates for $G = \GLr \times \GLr$.
    Then, for $n$ large enough, the equations in \eqref{mateq} can be
    expressed as:
    \begin{align}\label{mateq2}
      t^{\lambda_j} \sum_{k=0}^n g_{i,j}^{(k)} t^k
      \quad = \quad
      t^{\lambda_i} \sum_{k=0}^n h_{i,j}^{(k)} t^k
      \qquad
      \mod t^{n+1}.
    \end{align}

    Let $H\subset G$ be the truncation of \sta. Then $H$ is the subgroup of
    $G$ given by those equation in \eqref{mateq2} involving only the variables
    $g_{i,j}$ and $h_{i,j}$; these equations are:
    \begin{align}
      &g_{i,j} = h_{i,j}  && \text{if $\lambda_i = \lambda_j$,} \label{eq1}\\
      &g_{i,j} = 0  && \text{if $\lambda_i \ne \lambda_j$ and $i<j$,}
      \label{eq2}\\
      &h_{i,j} = 0  && \text{if $\lambda_i \ne \lambda_j$ and $i>j$.}
      \label{eq3}
    \end{align}
    Form \eqref{eq2} and \eqref{eq3}, we see that $H$ is a subgroup of
    $P_\lambda^{\text{op}} \times P_\lambda \subset G$, and \eqref{eq1} tells
    us that we can obtain $H$ from $P_\lambda^{\text{op}} \times P_\lambda$ by
    identifying the two copies of the Levi $L_\lambda$. Hence $[H] =
    [P_\lambda]^2 \, [L_\lambda]^{-1}$.
    
    From \eqref{mateq2} we also see that \sta is a sub-bundle of $H \times
    \g^{2n}$. More precisely, if $\h$ is the fiber of \sta over the identity
    in $H$, then $\h \subset \g^{2n}$ is an affine space and all the fibers of
    \sta are isomorphic to $\h$.  The codimension of $\h$ in $\g^{2n}$ can be
    computed with the same method used in the proof of Proposition
    \ref{discreps:stabcodim}: 
    \[
      \operatorname{codim}(\h,\g^{2n}) =
      nr^2 - \sum_{i=1}^r \lambda_i (2i-1) = nr^2 - b,
    \]
    where $b$ is the log discrepancy of the valuation induced by \orb\lambda.
    As a consequence
    \[
      [\h] \quad=\quad [\g^{2n}] \,\, \LL^{-nr^2+b} \quad=\quad \LL^{nr^2+b},
    \]
    and
    \begin{align*}
      \mu(\orb\lambda)
      &\,\,=\,\, 
      \LL^{r^2n}\, [\GLr]^2 \, [\sta]^{-1}
      \,\,=\,\, 
      \LL^{r^2n}\, [\GLr]^2 \, [H]^{-1} \, [\h]^{-1}
      \,\,=\,\, 
      \LL^{-b}\, [\GLr]^2 \, [H]^{-1}
      \\&\,\,=\,\, 
      \LL^{-b}\, [\GLr]^2 \, [P_\lambda]^{-2} \, [L_\lambda]
      \,\,=\,\, 
      \LL^{-b}\, [\GLr/P_\lambda]^2 \, [L_\lambda].
      \qedhere
    \end{align*}

  \end{proof}

  \subsection{Topological zeta function} 
  
  Recall from Section \ref{sec-motivic} that the motivic Igusa zeta function
  for the pair $(\M, \detvark)$ is defined as
  \[
    Z_{\detvark}(s)
    := \int_{\arc\M} |\II_{\detvark}|^s d\mu
    = \sum_{p=0}^\infty \mu\left(\operatorname{Cont}^{=p} \detvark\right)
    \LL^{-sp},
  \]
  where $\mu$ is the motivic measure on \arc\M and $\LL^{-s}$ is considered as
  a formal variable. The topological zeta function
  $Z^{\operatorname{top}}_{\detvark}(s)$ can be obtained from the Igusa zeta
  function by formally expanding $Z_{\detvark}(s)$ as a power series in
  $(\LL-1)$ and then extracting the constant term (i.e.\ by specializing $\LL$
  to $\chi(\AA^1)=1$).

  Using Propositions \ref{orbdec:cont} and \ref{prop:orbvolume} we can write
  $Z_{\detvark}(s)$ as
  \begin{equation}\label{eq:zeta1}
    Z_{\detvark}(s) 
    = \sum_{\lambda \in \ptr} 
      \mu(\orb\lambda) 
      \LL^{-s(\lambda_r+\dots+\lambda_{r-k})}
    = \sum_{\lambda \in \ptr} 
      [\GLr/P_\lambda]^2 [L_\lambda]
      \LL^{-b_\lambda-s(\lambda_r+\dots+\lambda_{r-k})},
  \end{equation}
  where $b_\lambda = \sum_{i=1}^r \lambda_i (2i-1)$ is the log discrepancy of
  the valuation induced by $\orb\lambda$. There are only finitely many
  possibilities for the value of $[\GLr/P_\lambda]^2 [L_\lambda]$, and it will
  be convenient to group the terms in the sum above accordingly. In order to
  do so, consider the bijection between $\ptr$ and $\NN^r$ given by
  \[
    \begin{gathered}
      \lambda = (\lambda_1, \lambda_2, \dots, \lambda_r) \in \ptr
      \quad\longmapsto\quad
      a(\lambda) = (\lambda_1-\lambda_2, \lambda_2-\lambda_3, \dots, \lambda_r)
      \in \NN^r,
    \\
      a = (a_1, a_2, \dots, a_r) \in \NN^r
      \quad\longmapsto\quad
      \lambda(a) = (a_1+\dots+a_r, a_2+\dots+a_r,\dots,a_r)
      \in \ptr.
    \end{gathered}
  \]
  For a subset $I \subseteq \{1,\dots,r-1\}$, consider $I^c =
  \{1,\dots,r-1\}\setminus I$ and define
  \[
    \Omega_I = \{\,\, a \in \NN^r \,\,:\,\, (a_i=0 \,\, \forall i\in I)
    \,\,\, \text{and} \,\,\, (a_j \ne 0 \,\, \forall j\in I^c) \,\,\}.
  \]
  Let $\lambda$ and $\lambda'$ be two partitions such that both $a(\lambda)$
  and $a(\lambda')$ belong to $\Omega_I$ for some $I$.  From the definitions
  of $P_\lambda$ and $L_\lambda$ (see Section \ref{sec:orbvolume}) we see that
  $[\GLr/P_\lambda]^2 [L_\lambda] = [\GLr/P_{\lambda'}]^2 [L_{\lambda'}]$.
  Hence, given a subset $I \subseteq \{1,\dots,r-1\}$ we can consider $\eta(I)
  = [\GLr/P_\lambda]^2 [L_\lambda]$, where $\lambda$ is any partition with
  $a(\lambda) \in \Omega_I$, and we obtain a well-defined function on the
  subsets of $\{1,\dots,r-1\}$. 
  
  Fix a subset $I \subseteq \{1,\dots,r-1\}$ and a partition $\lambda$ such
  that $a(\lambda)\in \Omega_I$. Consider $I^c_r = \{1,\dots,r\}\setminus I =
  \{i_1,\dots,i_\ell\}$, where $i_j < i_{j+1}$. Set $i_0=0$. Then
  $\GLr/P_\lambda$ is the manifold of partial flags of signature
  $(i_1,\dots,i_\ell)$, and its class in the Grothendieck group of varieties
  is given by
  \[
    [\GLr/P_\lambda] 
    = \prod_{j=1}^\ell [G(i_{j-1},i_j)]
    = \prod_{j=1}^\ell [G(i_j-i_{j-1},i_j)],
  \]
  where $G(u,v)$ is the Grassmannian of $u$-dimensional vector subspaces of
  $\CC^v$. Analogously:
  \[
    [L_\lambda] = \prod_{j=1}^\ell [\GL{i_j-i_{j-1}}].
  \]
  If we define $d(I,i_j) = i_j-i_{j-1}$ for $i_j \in I^c_r$, we can write:
  \begin{equation}\label{eq:zeta2}
    \eta(I) 
    = 
    [\GLr/P_\lambda]^2 [L_\lambda]
    =
    \prod_{i \not\in I} [G(d(I,i),i)]^2 [\GL{d(I,i)}].
  \end{equation}
  This shows more explicitly that $\eta(I)$ depends only on $I$, and not on the
  particular partition $\lambda$ in $\Omega_I$.
  From Equation \eqref{eq:zeta1} we obtain:
  \begin{equation}\label{eq:zeta3}
    Z_{\detvark}(s) 
    = \sum_{I\subseteq \{1,\dots,r-1\}}
      \left(
        \eta(I) 
        \sum_{a(\lambda)\in\Omega_I}
        \LL^{-b_\lambda-s(\lambda_r+\dots+\lambda_{r-k})}
      \right).
  \end{equation}
  Consider
  \[
    \psi(a)
    = b_{\lambda(a)}+s(\lambda(a)_r+\dots+\lambda(a)_{r-k})
    = \psi_1 a_1 + \dots + \psi_r a_r,
  \]
  where
  \begin{equation}\label{eq:psi}
    \psi_i = i^2 + s \,\, \max\{0, k+1+i-r\}.
  \end{equation}
  Then
  %\begin{multline}\label{eq:zeta4}
  %  \sum_{a(\lambda)\in\Omega_I} \LL^{-b_\lambda-s(\lambda_r+\dots+\lambda_{r-k})}
  %  =
  %  \sum_{a \in \Omega_I} \LL^{-\psi_1 a_1 - \dots - \psi_r a_r}
  %  =
  %  \prod_{i \in I} \left( \sum_{a_i=1}^\infty \LL^{-\psi_i a_i} \right)
  %  \cdot
  %  \prod_{i \not\in I} \left( \sum_{a_i=0}^\infty \LL^{-\psi_i a_i} \right)
  %  \\
  %  =
  %  \left( \prod_{i \in I} \frac{\LL^{-\psi_i}}{1-\LL^{-\psi_i}} \right)
  %  \cdot
  %  \left( \prod_{i \not\in I} \frac{1}{1-\LL^{-\psi_i}} \right)
  %  =
  %  \frac{\prod_{i \in I} \LL^{-\psi_i}}{\prod_{i=1}^r(1-\LL^{-\psi_i})}
  %  =
  %  \frac{\prod_{i \not\in I} \LL^{\psi_i}}{\prod_{i=1}^r(\LL^{\psi_i}-1)}.
  %\end{multline}
  \begin{multline}\label{eq:zeta4}
    \sum_{a(\lambda)\in\Omega_I} \LL^{-b_\lambda-s(\lambda_r+\dots+\lambda_{r-k})}
    \quad = \quad
    \sum_{a \in \Omega_I} \LL^{-\psi_1 a_1 - \dots - \psi_r a_r}
    \quad = \quad
    \sum_{a \in \Omega_I} \LL^{-\sum_{i\not\in I}\psi_i a_i}
    \\
    = \quad
    \left(
    \mathop{\prod_{i \not\in I}}_{i \ne r} \,\, \sum_{a_i=1}^\infty \LL^{-\psi_i a_i}
    \right) \cdot \left(
    \sum_{a_r=0}^\infty \LL^{-\psi_r a_r}
    \right)
    \quad = \quad
    \LL^{\psi_r} \cdot
    \prod_{i \not\in I} \frac{\LL^{-\psi_i}}{1-\LL^{-\psi_i}}
    \quad = \quad
    \LL^{\psi_r} \cdot
    \prod_{i \not\in I} \frac{1}{\LL^{\psi_i} - 1}.
  \end{multline}
  Combining Equations \eqref{eq:zeta2}, \eqref{eq:zeta3}, and
  \eqref{eq:zeta4}, we get:
  %\[
  %  Z_{\detvark}(s) 
  %  =
  %  \frac{1}{\prod_{i=1}^r(\LL^{\psi_i}-1)} \cdot
  %  \sum_{I \subseteq \{1,\dots,r-1\}}
  %  \prod_{i \not\in I}
  %  \LL^{\psi_i}
  %  [G(d(I,i),i)]^2 [\GL{d(I,i)}].
  %\]
  \begin{equation}\label{eq:zeta5}
    Z_{\detvark}(s) 
    =
    \LL^{\psi_r}
    \sum_{I \subseteq \{1,\dots,r-1\}}
    \prod_{i \not\in I}
    \frac{1}{\LL^{\psi_i}-1}
    [G(d(I,i),i)]^2 [\GL{d(I,i)}].
  \end{equation}
  We will not try to simplify Equation \eqref{eq:zeta5} any further.
  Instead, we will use it to compute the topological zeta function. As
  explained in Section \ref{sec-motivic}, it is enough to expand each
  summand in \eqref{eq:zeta5} as a power series in $(\LL-1)$. We will write
  $O((\LL-1)^\ell)$ to denote a power series in $(\LL-1)$ which is divisible
  by $(\LL-1)^\ell$. We have:
  \[
      \LL^{\psi_i} = 1 + O(\LL-1), \qquad
      \LL^{\psi_i}-1 = \psi_i \cdot (\LL-1) + O((\LL-1)^2),
      %&[\GL1] = \LL-1, &
      %&[\GL d] = O((\LL-1 )^d),
  \]
  and
  \[
    [G(1,i)] = [\PP^{i-1}] = 1+\LL+\LL^2+\dots+\LL^{i-1} = i + O(\LL-1).
  \]
  The class of $\GL d$ can be computed by induction on $d$. Let $P =
  P_{(1)}$ be the parabolic subgroup of $\GL d$ whose corresponding
  quotient is projective space $\GL d/P = \PP^{d-1}$. Then, for $d>1$,
  \[
    [P] = [\GL 1] [\GL {d-1}] \LL^{d-1}
    = [\GL 1] [\GL {d-1}] (1 + O(\LL-1)),
    \qquad
    [\PP^{d-1}] = d + O(\LL-1).
  \]
  Since $[\GL 1] = \LL -1$, we see that
  \[
     [\GL d] = [P] [\PP^{d-1}] = [\GL 1] [\GL {d-1}] (d + O(\LL-1)) = O((\LL-1 )^d).
  \]
  Hence
  \[
    \frac{1}{\LL^{\psi_i}-1}
    [G(d,i)]^2 [\GL d] = O((\LL-1)^{d-1}),
  \]
  and
  \[
    \frac{1}{\LL^{\psi_i}-1}
    [G(1,i)]^2 [\GL 1] =  \frac{i^2}{\psi_i} + O(\LL-1).
  \]
  In particular, the only summands in Equation \eqref{eq:zeta5} not divisible
  by $(\LL - 1)$ are those for which $d(I,i)=1$ for all $i\not\in I$. Since
  $d(I,i)=1$ if and only if $i-1 \not\in I$, the only significant summand is
  the one corresponding to $I = \emptyset$. Hence
  \[
    Z_{\detvark}(s) 
    =
    \prod_{i=1}^r \frac{i^2}{\psi_i} + O(\LL-1).
  \]
  Combining this with Equation \eqref{eq:psi} we get the topological zeta
  function:
  \begin{multline*}
    Z^{\operatorname{top}}_{\detvark}(s) 
    \quad = \quad
    \prod_{i=1}^r \frac{i^2}{\psi_i}
    \quad = \quad
    \prod_{i=1}^{r-k-1}\frac{i^2}{i^2}
    \prod_{i=r-k}^{r}\frac{i^2}{i^2+s(k+1-i-r)}
    \\=\quad
    \prod_{i=r-k}^{r}
    \left(1+s \, \frac{k+1-i-r}{i^2}\right)^{-1}
    \quad = \quad
    %\prod_{j=0}^{k}\frac{1}{(r-j)^2+s(k+1-j)}.
    \prod_{j=0}^{k}\left(1+s \, \frac{k+1-j}{(r-j)^2}\right)^{-1}.
  \end{multline*}

  The following theorem summarizes the results of this section.

  %\begin{theorem}
  %  Let $\M = \AA^{r^2}$ be the space of square $r \times r$ matrices, and let
  %  \detvark be the subvariety of matrices of rank at most $k$. Then the
  %  topological zeta function of the pair $(\M, \detvark)$ is given by
  %  \[
  %    Z^{\operatorname{top}}_{\detvark}(s) 
  %    =
  %    \prod_{j=0}^{k}\left(1+s \, \frac{k+1-j}{(r-j)^2}\right)^{-1}.
  %  \]
  %  In particular, the poles of the topological zeta function are
  %  \[
  %    -\frac{r^2}{k+1},\quad
  %    -\frac{(r-1)^2}{k},\quad
  %    -\frac{(r-2)^2}{k-1},\quad
  %    ...,\quad
  %    -(r-k)^2.
  %  \]
  %\end{theorem}

  \begin{theorem}
    Let $\M = \AA^{r^2}$ be the space of square $r \times r$ matrices, and let
    \detvark be the subvariety of matrices of rank at most $k$. Then the
    topological zeta function of the pair $(\M, \detvark)$ is given by
    \[
      Z^{\operatorname{top}}_{\detvark}(s) 
      =
      \prod_{\zeta \in \Omega} \frac{1}{1-s\,\zeta^{-1}}
    \]
    where $\Omega$ is the set of poles:
    \[
      \Omega = \left\{
      \,\,
      -\frac{r^2}{k+1},\quad
      -\frac{(r-1)^2}{k},\quad
      -\frac{(r-2)^2}{k-1},\quad
      \dots,\quad
      -(r-k)^2
      \,\,
      \right\}.
    \]
  \end{theorem}

  \bibliography{references_detarcs}

\providecommand{\bysame}{\leavevmode\hbox to3em{\hrulefill}\thinspace}
\providecommand{\MR}{\relax\ifhmode\unskip\space\fi MR }
% \MRhref is called by the amsart/book/proc definition of \MR.
\providecommand{\MRhref}[2]{%
  \href{http://www.ams.org/mathscinet-getitem?mr=#1}{#2}
}
\providecommand{\href}[2]{#2}
\begin{thebibliography}{dCEP80}

\bibitem[BV88]{BV88}
Winfried Bruns and Udo Vetter, \emph{Determinantal rings}, Lecture Notes in
  Mathematics, vol. 1327, Springer-Verlag, Berlin, 1988.

\bibitem[dCEP80]{dCEP80}
C.~de~Concini, David Eisenbud, and C.~Procesi, \emph{Young diagrams and
  determinantal varieties}, Invent. Math. \textbf{56} (1980), no.~2, 129--165.

\bibitem[dFEI08]{dFEI08}
Tommaso de~Fernex, Lawrence Ein, and Shihoko Ishii, \emph{Divisorial valuations
  via arcs}, Publ. Res. Inst. Math. Sci. \textbf{44} (2008), no.~2, 425--448.

\bibitem[DL98]{DL98}
Jan Denef and Fran{\c{c}}ois Loeser, \emph{Motivic {I}gusa zeta functions}, J.
  Algebraic Geom. \textbf{7} (1998), no.~3, 505--537.

\bibitem[DL99]{DL99}
\bysame, \emph{Germs of arcs on singular algebraic varieties and motivic
  integration}, Invent. Math. \textbf{135} (1999), no.~1, 201--232.

\bibitem[DL02]{DL02}
\bysame, \emph{Motivic integration, quotient singularities and the {M}c{K}ay
  correspondence}, Compositio Math. \textbf{131} (2002), no.~3, 267--290.

\bibitem[ELM04]{ELM04}
Lawrence Ein, Robert Lazarsfeld, and Mircea Musta{\c{t}}{\u{a}}, \emph{Contact
  loci in arc spaces}, Compos. Math. \textbf{140} (2004), no.~5, 1229--1244.

\bibitem[EM06]{EM06}
Lawrence Ein and Mircea Musta{\c{t}}{\u{a}}, \emph{Invariants of singularities
  of pairs}, International {C}ongress of {M}athematicians. {V}ol. {II}, Eur.
  Math. Soc., Z\"urich, 2006, pp.~583--602.

\bibitem[EMY03]{EMY03}
Lawrence Ein, Mircea Musta{\c{t}}{\u{a}}, and Takehiko Yasuda, \emph{Jet
  schemes, log discrepancies and inversion of adjunction}, Invent. Math.
  \textbf{153} (2003), no.~3, 519--535.

\bibitem[Ful97]{Ful97}
William Fulton, \emph{Young tableaux}, London Mathematical Society Student
  Texts, vol.~35, Cambridge University Press, Cambridge, 1997, With
  applications to representation theory and geometry.

\bibitem[GS06]{GS06}
Russell~A. Goward, Jr. and Karen~E. Smith, \emph{The jet scheme of a monomial
  scheme}, Comm. Algebra \textbf{34} (2006), no.~5, 1591--1598.

\bibitem[Ish04]{Ish04}
Shihoko Ishii, \emph{The arc space of a toric variety}, J. Algebra \textbf{278}
  (2004), no.~2, 666--683.

\bibitem[Ish08]{Ish06}
\bysame, \emph{Maximal divisorial sets in arc spaces}, Algebraic geometry in
  {E}ast {A}sia---{H}anoi 2005, Adv. Stud. Pure Math., vol.~50, Math. Soc.
  Japan, Tokyo, 2008, pp.~237--249.

\bibitem[Joh03]{Joh03}
Amanda~Ann Johnson, \emph{Multiplier ideals of determinantal ideals}, Ph.D.
  thesis, University of Michigan, 2003.

\bibitem[Kon95]{Kon95}
Maxim Kontsevich, \emph{Motivic integration}, 1995, Lecture at Orsay.

\bibitem[KS05a]{KS05a}
Toma{\v{z}} Ko{\v{s}}ir and B.~A. Sethuraman, \emph{Determinantal varieties
  over truncated polynomial rings}, J. Pure Appl. Algebra \textbf{195} (2005),
  no.~1, 75--95.

\bibitem[KS05b]{KS05b}
\bysame, \emph{A {G}roebner basis for the {$2\times2$} determinantal
  ideal{$\!\!\mod t\sp 2$}}, J. Algebra \textbf{292} (2005), no.~1, 138--153.

\bibitem[Lak87]{Lak85}
Dan Laksov, \emph{Completed quadrics and linear maps}, Algebraic geometry,
  {B}owdoin, 1985 ({B}runswick, {M}aine, 1985), Proc. Sympos. Pure Math.,
  vol.~46, Amer. Math. Soc., Providence, RI, 1987, pp.~371--387.

\bibitem[LJ90]{LJ90}
Monique Lejeune-Jalabert, \emph{Courbes trac\'ees sur un germe d'hypersurface},
  American Journal of Mathematics \textbf{112} (1990), no.~4, 525--568.

\bibitem[LJR98]{LJR98}
Monique Lejeune-Jalabert and Ana Reguera, \emph{Arcs and wedges on sandwiched
  surface singularities}, C. R. Acad. Sci. Paris S\'er. I Math. \textbf{326}
  (1998), no.~2, 207--212.

\bibitem[LJR08]{LJR08}
\bysame, \emph{Exceptional divisors which are not uniruled belong to the image
  of the nash map}, Arxiv preprint arXiv:0811.2421 (2008).

\bibitem[Mus01]{Mus01}
Mircea Musta{\c{t}}{\u{a}}, \emph{Jet schemes of locally complete intersection
  canonical singularities}, Invent. Math. \textbf{145} (2001), no.~3, 397--424,
  With an appendix by David Eisenbud and Edward Frenkel.

\bibitem[Mus02]{Mus02}
\bysame, \emph{Singularities of pairs via jet schemes}, J. Amer. Math. Soc.
  \textbf{15} (2002), no.~3, 599--615 (electronic).

\bibitem[Nas95]{Nas95}
John~F. Nash, Jr., \emph{Arc structure of singularities}, Duke Math. J.
  \textbf{81} (1995), no.~1, 31--38 (1996), A celebration of John F. Nash, Jr.

\bibitem[Nob91]{Nob91}
Augusto Nobile, \emph{On {N}ash theory of arc structure of singularities}, Ann.
  Mat. Pura Appl. (4) \textbf{160} (1991), 129--146 (1992).

\bibitem[Pl{\'e}05]{Ple05}
Camille Pl{\'e}nat, \emph{R\'esolution du probl\`eme des arcs de {N}ash pour
  les points doubles rationnels {$D\sb n$}}, C. R. Math. Acad. Sci. Paris
  \textbf{340} (2005), no.~10, 747--750.

\bibitem[PPP06]{PPP06}
Camille Pl{\'e}nat and Patrick Popescu-Pampu, \emph{A class of non-rational
  surface singularities with bijective {N}ash map}, Bull. Soc. Math. France
  \textbf{134} (2006), no.~3, 383--394.

\bibitem[Sem51]{Sem51}
John~G. Semple, \emph{The variety whose points represent complete collineations
  of {$S\sb r$} on {$S'\sb r$}}, Univ. Roma. Ist. Naz. Alta Mat. Rend. Mat. e
  Appl. (5) \textbf{10} (1951), 201--208.

\bibitem[Tyr56]{Tyr56}
J.~A. Tyrrell, \emph{Complete quadrics and collineations in {$S\sb n$}},
  Mathematika \textbf{3} (1956), 69--79.

\bibitem[Vai84]{Vai84}
Israel Vainsencher, \emph{Complete collineations and blowing up determinantal
  ideals}, Math. Ann. \textbf{267} (1984), no.~3, 417--432.

\bibitem[Vey06]{Vey06}
Willem Veys, \emph{Arc spaces, motivic integration and stringy invariants},
  Singularity theory and its applications, Adv. Stud. Pure Math., vol.~43,
  Math. Soc. Japan, Tokyo, 2006, pp.~529--572.

\bibitem[Voj07]{Voj07}
Paul Vojta, \emph{Jets via {H}asse-{S}chmidt derivations}, Diophantine
  geometry, CRM Series, vol.~4, Ed. Norm., Pisa, 2007, pp.~335--361.

\bibitem[Yue07a]{Yue07a}
Cornelia Yuen, \emph{Jet schemes of determinantal varieties}, Algebra, geometry
  and their interactions, Contemp. Math., vol. 448, Amer. Math. Soc.,
  Providence, RI, 2007, pp.~261--270.

\bibitem[Yue07b]{Yue07b}
\bysame, \emph{The multiplicity of jet schemes of a simple normal crossing
  divisor}, Comm. Algebra \textbf{35} (2007), no.~12, 3909--3911.

\end{thebibliography}
  \bibliographystyle{amsalpha}

  \end{document}